\documentclass[a4paper]{amsart}

\usepackage{fullpage}
\usepackage{amsmath}
\usepackage{hyperref}

\usepackage[initials,nobysame,msc-links]{amsrefs}
\DefineSimpleKey{bib}{how}
\renewcommand{\eprint}[1]{#1}

\BibSpec{misc}{%
  +{}{\PrintAuthors}  {author}
  +{,}{ \textit}      {title}
  +{.}{ }             {how}
  +{}{ \parenthesize} {date}
  +{,} { available at \eprint}        {eprint}
  +{,}{ }             {note}
  +{.}{}              {transition}
}

\usepackage{amssymb,MnSymbol}
\usepackage{amscd}  % easy commutative diagram
\usepackage[all]{xy} % complex commutative diagram and else

\usepackage{amsthm}
\theoremstyle{plain}
\newtheorem{thm}[]{Theorem}
\newtheorem{prop}[]{Proposition}

\newtheorem{cor}[prop]{Corollary}

\theoremstyle{definition}

\theoremstyle{remark}
\newtheorem{rmk}[prop]{Remark}
\newtheorem{example}[prop]{Example}

\newcommand{\ensemble}[1]{\left\{ #1 \right\}}

\newcommand{\absolute}[1]{\left| #1 \right|}
\newcommand{\pairing}[2]{\left \langle #1, #2 \right\rangle}

\newcommand{\schwfuncs}{\mathcal{S}}
\newcommand{\N}{\mathbb{N}}
\newcommand{\Z}{\mathbb{Z}}

\newcommand{\C}{\mathbb{C}}
\newcommand{\R}{\mathbb{R}}

\newcommand{\cycl}{\mathrm{cyc}}
\newcommand{\triv}{\mathrm{triv}}
\newcommand{\smooth}[1]{\mathcal{#1}}
\newcommand{\opos}{\mathrm{op}}
\newcommand{\Lin}{\mathrm{Hom}_\C}
\newcommand{\U}{\mathrm{U}}
\newcommand{\clsB}{B}
\newcommand{\bBcplx}{\mathcal{B}}
\newcommand{\nbBcplx}{\bar{\bBcplx}}
\newcommand{\nC}{\bar{C}}
\newcommand{\Tot}{\mathrm{Tot}}
\newcommand{\propEx}{\underline{E}}

\DeclareMathOperator{\HP}{\mathrm{HP}}
\DeclareMathOperator{\HC}{\mathrm{HC}}
\DeclareMathOperator{\CC}{\mathrm{CC}}

\DeclareMathOperator{\HH}{\mathrm{HH}}

\DeclareMathOperator{\bop}{\mathbf{b}}
\DeclareMathOperator{\Bop}{\mathbf{B}}
\DeclareMathOperator{\Lop}{\mathcal{L}}

\DeclareMathOperator{\End}{End}

\DeclareMathOperator{\Ext}{Ext}

\DeclareMathOperator{\ch}{ch}

\DeclareMathOperator{\ev}{ev}

\DeclareMathOperator{\Sh}{Sh}
\DeclareMathOperator{\sh}{sh}

\DeclareMathOperator{\ptensor}{\hat{\otimes}}

\makeatletter
\def\blfootnote{\gdef\@thefnmark{}\@footnotetext}
\makeatother

\begin{document}
\title{Monodromy of the Gauss--Manin connection for deformation by group cocycles} % short: [Gauss--Manin for cocycle deformation]
\author{Makoto Yamashita}
\address{Department of Mathematics, Ochanomizu University\\
Otsuka 2-1-1, Bunkyo\\
112-8610, Tokyo, Japan}
\email{yamashita.makoto@ocha.ac.jp}
\date{September 21, 2012; February 8, 2017} % Activate to display a given date or no date

\begin{abstract}
We consider the $2$-cocycle deformation of algebras graded by discrete groups.  The action of the Maurer--Cartan form on cyclic cohomology is shown to be cohomologous to the cup product action of the group cocycle.  This allows us to compute the monodromy of the Gauss--Manin connection in the strict deformation setting.%  In the case of twisted group algebras, we obtain a purely algebraic explanation of the `range of trace' formula by Mathai.
\end{abstract}

\maketitle

\section{Introduction}

This paper concerns the cyclic cohomology of strict deformation associated with group cocycles.  This forms a homological algebraic counterpart of our previous paper on the C$^*$-algebraic K-theory of such deformations~\cite{arXiv:1107.2512}.  We consider the algebras which admit grading by discrete groups (also known as Fell bundles), where the deformation parameter is given by $\U(1)$-valued $2$-cocycles on the structure group.

One important example of such a deformation is the noncommutative torus, which is $\Z^2$-graded by the Fourier decomposition.  As is well known, its K-group is isomorphic to $\Z^2$, the same as that of classical $2$-torus.  The deformation parameter appears when one considers the functional on the $K_0$-group induced by the standard trace~\citelist{\cite{rieffel-irr-rot-pres}\cite{MR595412}}.  This result was later generalized to the higher dimensional noncommutative tori by Elliott~\cite{MR731772}, and to the reduced twisted group algebra $C^*_{r, \omega}(\Gamma)$ for a certain class of discrete groups by Mathai~\cite{MR2218025}.

The K-theory isomorphism results for the above continuous deformations can be shown following a common pattern: given a continuous (often smooth) family of pre-C$^*$-algebra structure $(m_t)_{t \in I}$ on a single vector space $A$, one forms the C$^*$-algebra $A_I$ of `global sections' for the bundle of C$^*$-algebras $A_t = \overline{(A; m_t)}$ for $t \in I$.  Then one shows that the evaluation homomorphism $A_I \rightarrow A_t$ induces an isomorphism in the K-theory for any $t$, from which we obtain the natural isomorphism of the groups $K_*(A_t)$ for the different values of $t \in I$.

Since the above construction only uses algebra homomorphisms, one may mimic the construction at the level of cyclic homology groups of smooth subalgebras.  Then, we can understand how the cyclic cocycles $\HC^*(A_t)$ pair with the groups $K_*(A_t)$ and $\HP_*(A_t)$, by identifying the families $(\phi_t \in \HC^*(A_t))_{t \in I}$ which become constant after pulling back to $A_I$.  The global section algebra $A_I$ represents the `total space' of the bundle of `noncommutative spaces' represented by $(A_t)_{t \in I}$, and the evaluation maps $\HP_*(A_I) \rightarrow \HP_*(A_t)$ can be considered as the restriction of cohomology classes to the fiber at $t$.  The Gauss--Manin connection on the periodic cyclic theory introduced by Getzler~\cite{MR1261901} gives the infinitesimal parallel transport operation on the family $(\HP_*(A_t))_{t \in I}$ such that the image of $\HP_*(A_I)$ becomes flat.

Although the Gauss--Manin connection form has a simple presentation, the monodromy operator could be rather complicated~\citelist{\cite{MR2308582}\cite{MR2762543}} in general.  Since it is a priori expressed as an infinite series of operators on the infinite dimensional vector space of cyclic chains, there is a issue of convergence when one tries to work on the strict deformation setting.  Thus, in order to compute it (even to make sense of it), one needs to work on an additional structure such as Poisson manifolds which gives us a nicer cohomology group in which the Maurer--Cartan form resides~\citelist{\cite{MR2444365}*{Sections~1.33--36}\cite{arXiv:0906.3122}}.

Our main result is that, for $2$-cocycle deformation of Fell bundles, the action of the Gauss--Manin connection on cyclic cohomology classes is identified with the action of the group $2$-cocycle via the cup product operation (Theorem~\ref{thm:gauss-manin-monodromy}).  Because of this translation, we are able to integrate the action and construct the monodromy operator in a strict deformation setting.  In the case of twisted group algebras, the resulting formula is comparable to that of Mathai.

The key machinery of the proof is Karoubi's embedding of group cohomology into the cyclic cohomology of the group algebra.  On the one hand this embedding is a ring homomorphism with respect to the cup product on the cyclic cohomology (Section~\ref{sec:alg-str-on-cyc-coh-of-grp}).  On the other hand, the action of Gauss--Manin connection is identified with that of the group $2$-cocycle parametrizing the deformation.  Thus, the integration of the connection makes sense as the exponentiation of the $2$-cocycle in the group cohomology ring.  We remark that the cup product operation in a more general context of Hopf-cyclic theory is an actively developing subject~\citelist{\cite{MR2112033}\cite{MR2475613}\cite{arXiv:1011.2735}}.

Most of the results in this paper should be true for the fell bundles over locally compact groups after reasonable modifications, but we restrict ourselves to the case of $\R^n$-algebras which is handled in the last part of the paper.  Even without Fourier decomposition, we can directly handle the action of generators of the action in an algebraic way as mutually commuting derivations.  In this case, the exponentiation of the curvature takes a simpler form because the square of interior product of a derivation is cohomologically trivial.  As a consequence, we recover the Connes--Thom isomorphism in cyclic cohomology~\cite{MR945014} for the case of invariant cocycles.

\paragraph{Acknowledgment}  The author thanks Ryszard Nest, whose suggestion to look into the theory of Gauss--Manin connection was the reason this project came into existence.  He is also thankful to Sergey Neshveyev and Catherine Oikonomides for illuminating discussions.

\section{Preliminaries}

Throughout the paper $I$ denotes the closed interval $[-1, 1]$.  When $V$ is a vector space, we denote its linear dual by $V'$.  When there is no fear of confusion, the tensor product $a^0 \otimes \cdots \otimes a^n$ is written as $(a^0, \ldots, a^n)$.

Let $A$ be an algebra over $\C$.  We let $A^+$ denote its unitization, given by $A \oplus \C$ endowed with the product structure $(a, \lambda) . (b, \mu) = (ab + \mu a + \lambda b, \lambda \mu)$.  For notational simplicity we write $a + \lambda = (a, \lambda)$ when there is no fear of confusion.  The enveloping algebra $A \otimes A^\opos$ is denoted by $A^e$.  Thus, an $A$-bimodule is the same thing as an $A^e$-module.

Although most of the homological constructions in this paper can be carried out for differential graded algebras or even more general $A_\infty$-algebras, we stick to the case of ungraded algebras.  We also use the $\Z/2\Z$-gradings instead of the $\Z$-gradings for the periodic theory.

\subsection{Cyclic homology theories}

We follow the convention of Loday~\cite{MR1600246} for the most part on the fundamental concepts of cyclic homology theory, such as
\begin{itemize}
\item a \textit{cyclic set} $(X_*, d_*, s_*, t_*)$, which is a contravariant functor from the cyclic category $\Delta^\cycl$ to the category of sets,
\item a \textit{cyclic module} $(C_*, d_*, s_*, \lambda_*)$, which is a contravariant functor from $\Delta^\cycl$ to the category of vector spaces,
\item the \textit{cyclic bicomplex} $\CC_*(C)$ associated with a cyclic module $C_*$,
\item and a \textit{mixed module} $(C_*, b, B)$.
\end{itemize}

When $C_*$ is a cyclic module, there are two associated mixed modules, the \textit{$(b, B)$-bicomplex} $\bBcplx_*(C)$ and the \textit{normalized $(b, B)$-bicomplex} $\nbBcplx_*(C)$.  The first one is given by
\[
\begin{split}
b_n &= \sum_{j=0}^{n} (-1)^j d_j \colon C_n \rightarrow C_{n-1}\\
B_n &= (1 - \lambda_{n+1}) \lambda_{n+1} s_n \sum_{j=0}^n \lambda_n^j \colon C_n \rightarrow C_{n+1}.
\end{split}
\]
The second one is formed by the groups
\[
\nC_n = C_n / \left (s_0(C_{n-1}) + \cdots + s_{n-1}(C_{n-1}) \right )\\
\]
and the $b$, $B$ operators induced by the ones of $\bBcplx_*(C)$.

When $C = (C_*, b, B)$ is a mixed module, we let $(C^n)_{n \in \N}$ denote the dual complex $(C_n', b^t, B^t)$ and $\HC^*(C)$ be the cohomology group of $(\Tot C^*, b^t - B^t)$.  The differential is different from Loday's convention $b + B$, but should not bring in any fundamental difference, the $S$-morphisms have modified components when we express them as matrices.

When $C_*$ is a cyclic module, we write $\HC^*(C) = \HC^*(\bBcplx(C))$ when there is no fear of confusion.  In that case, the subcomplex $\nbBcplx(C)' \subset \bBcplx(C)'$ gives the same cyclic cohomology group.  The cyclic bicomplex
\[
\CC^*(C) = (C_{m-n}', b^t, (b')^t, 1-\lambda, \sum \lambda^j)
\]
and the cyclic cochain complex
\[
C_\lambda^* = \left ( \ker \left ( 1 - \lambda_n \right ), b^t \right )_{n\in \N} \subset (C^*, b^t)
\]
also compute the cyclic cohomology group of $C$.

\begin{example}
Let $A$ be a unital algebra.  The associated cyclic module $C_*(A)$ is given by $C_n(A) = A^{\otimes n + 1}$, and the normalized $(b, B)$-bicomplex is formed by by $\nC_n(A) = A \otimes (A/\C)^{\otimes n}$.  As usual, we put $\Omega^0(A) = A$ and $\Omega^n(A) = A^+ \otimes A^{\otimes n}$.  Thus, the cyclic bicomplex $\CC_*(A)$ is a direct sum of the $\Omega^n(A)$.  The definition of cyclic bicomplex makes sense for nonunital algebras as well (although there is no longer a cyclic module behind it), and it can be identified with the normalized $(b, B)$-bicomplex of $A^+$ except for the first row, where they differ by a direct summand $\C$ at every even column.
\end{example}

\begin{example}
Let $\Gamma$ be a discrete group.  Its nerve cyclic set $\clsB_* \Gamma$ is given by $\clsB_n \Gamma = \Gamma^{n}$ endowed the simplicial structure of nerve of $\Gamma$ regarded as a category with one object, and the cyclic structure
\begin{align*}
t_n(g_1, \ldots, g_n) &= ((g_1 \cdots g_n)^{-1}, g_1, \ldots, g_{n-1}).
\end{align*}
An $n$-cocycle $\phi$ in $\C[\clsB_n \Gamma]'$ is normalized when it satisfies
\begin{equation*}
% \label{eq:grp-cocycl-normalization}
\phi(g_1, \ldots, g_{j-1}, e, g_{j+1}, \cdots, g_n) = 0
\end{equation*}
for all $1 \le j \le n$.

On one hand, the associated cyclic module $\C[\clsB_*\Gamma]$ is equal to the standard resolution of the trivial $\Gamma$-module $\C$.  On the other hand, it can be regarded as a direct summand of $C_*(\C[\Gamma])$ via the morphism
\[
\C[\clsB_n \Gamma] \rightarrow C_n(\C[\Gamma]), \quad (g_1, \ldots, g_n) \mapsto ((g_1 \cdots g_n)^{-1}, g_1, \ldots, g_n).
\]
This induces Karoubi's inclusion~\cite{MR732839}
\[
H^n(\Gamma) \oplus H^{n-2}(\Gamma) \oplus \cdots \oplus H^{n - 2 \lfloor n / 2 \rfloor}(\Gamma) \simeq \HC^n(\clsB_* \Gamma) \subset \HC^n(\C[\Gamma])
\]
denoted by $\phi \mapsto \tilde{\phi}$.  For example, the unit class $1 \in H^0(\Gamma)$ is mapped to the class standard trace $\tau(g) = \delta_{e, g}$ under this correspondence.
\end{example}

\subsection{Cup product in cyclic cohomology}

When $C_*$ and $D_*$ are mixed complexes, we obtain a new mixed complex $C \times D$ given by
\[
(C \times D)_n = C_n \otimes D_n,\quad b^{C \times D}_n = b^C_n \otimes b^D_n,\quad B^{C \times D}_n = B^C_n \otimes B^D_n,
\]
and another one $C \otimes D$ given by
\begin{align*}
(C \otimes D)_n &= \bigoplus_{p+q=n} C_p \otimes C_q,\\
b^{C \otimes D}_n &= \bigoplus_{p, q} b_p \otimes 1 + (-1)^p \otimes b_q,\\
B^{C \otimes D}_n &= \bigoplus_{p, q} B_p \otimes 1 + (-1)^p \otimes B_q.
\end{align*}
When $C$ and $D$ are cyclic modules, we write $C \times D, C \otimes D$ instead of $\bBcplx_* C \times \bBcplx_* D, \bBcplx_* C \otimes \bBcplx_* D$ when there is no fear of confusion.

When $(p, q)$ is a pair of nonnegative integers, a \textit{$(p, q)$-shuffle} is a partition of $\ensemble{1, \ldots, p + q}$ into two subsets $\ensemble{m_1 < \ldots < m_p}, \ensemble{n_1 < \ldots < n_q}$ of cardinality $p$ and $q$.  To such a shuffle we associate the element $\sigma$ of $S_{p + q}$ determined by
\[
\sigma(x) = 
\begin{cases}
m_x & (1 \le x \le p)\\
n_{x-p} & (p < x \le p + q).
\end{cases}
\]
We let $S_{p, q}$ denote the set of $(p, q)$-shuffles, and identify it with a subset of $S_{p + q}$ by the above correspondence.

Suppose that $C_*$ and $D_*$ are cyclic modules.  Let $p + q = n$, and let $\sh_{p, q}$ denote the linear map of $C_p \otimes D_q$ into $C_n \otimes D_n$ defined by
\[
\sh_{p, q}(a, b) = \sum_{\sigma \in S_{p, q}} s_{\sigma(p+1)} \cdots s_{\sigma(p+q)}(a) \otimes s_{\sigma(1)} \cdots s_{\sigma(p)}(b).
\]
The direct sum of these maps for $p + q = n$ defines a map $(C \otimes D)_n \rightarrow (C \times D)_n$ denoted by $\sh_n$.

Next, a \textit{$(p, q)$-cyclic shuffle} is any element $\sigma$ of $S_{p + q}$ which can be written as $\sigma = \mu \circ t_{1\cdots p}^a t_{(p+1)\cdots(p+q)}^b$ for a $(p, q)$-shuffle $\mu$ and $a, b \in \N$, which also satisfies $\sigma(1) < \sigma(p+1)$.  Here, $t_{1\cdots p}$ is the cyclic permutation $t_{1\cdots p}(k) = k\pmod{p}$ on $\ensemble{1, \ldots, p}$, and $t_{(p+1)\cdots(p+q)}$ is the one $t_{(p+1)\cdots(p+q)}(k) = k+1\pmod{q}$ on $\ensemble{p+1, \ldots, p+q}$.  We put $S^\cycl_{p, q}$ the set of $(p, q)$-cyclic shuffles.

For cyclic modules $C_*$ and $D_*$, one obtains a map $\perp_{p, q}$ from $C_p \otimes D_q$ to $C_{p+q} \otimes D_{p+q}$ given by
\[
x \otimes y \mapsto \sum_{\sigma = \mu t_{1\cdots p}^a t_{(p+1)\cdots(p+q)}^b\in S^\cycl_{p, q}} (-1)^{|s|} s_{\mu(p+1)} \cdots s_{\mu(p+q)} t^a(x) \otimes s_{\mu(1)} \cdots s_{\mu(p)} t^b (y).
\]
We then obtain a linear map $\sh'_{p, q}$ of $C_p \otimes D_q$ into $C_{n+2} \otimes D_{n+2}$ by
\[
\sh'_{p, q}(a, b) = s(a) \perp_{p+1, q+1} s(b),
\]
where $s$ is the extra degeneracy operator given by $t_{p+1} s_p\colon C_p \rightarrow C_{p+1}$ and similarly for $D_*$.  Then,  the direct sum of the $\sh'_{p, q}$ for $p + q = n$ defines a map $(C \otimes D)_n \rightarrow (C \times D)_{n+2}$ denoted by $\sh'_n$.

The map $\Sh = \sh - \sh'$ from $\Tot \bBcplx(C \otimes D)$ to $\Tot \bBcplx(C \times D)$ is a morphism of the total complexes of the mixed complexes, which is actually a quasi-isomorphism~\cite{MR1600246}*{Thteorem~4.3.8}.

Consider the transpose map $\Sh^t$ between the duals of mixed complexes, and let
\[
\Psi\colon \HC^*(C \otimes D) \rightarrow \colon \HC^*(C \times D)
\]
denote the inverse of the induced map $\HC^*(\Sh^t)$.  Via the standard identification $\HC^*(C) \otimes \HC^*(D) \simeq \HC^*(C \otimes D)$, the map $\Psi$ induces the external cup product
\[
\cup\colon \HC^*(C) \otimes \HC^*(D) \rightarrow \HC^*(C \times D).
\]

Suppose that $A$ and $B$ are unital algebras.  We have a natural identification $\bBcplx_*(A) \times \bBcplx_*(B) = \bBcplx_*(A \otimes B)$.  Hence the above cup product becomes
\[
\cup\colon \HC^*(A) \otimes \HC^*(B) \rightarrow \HC^*(A \otimes B).
\]

There is a similar product structure on the Hochschild cohomology group~\cite{MR0077480}*{Chapter~XI},
\[
\vee\colon \Ext_{A^e}(A, A') \times \Ext_{B^e}(B, B') \rightarrow \Ext_{(A \otimes B)^e}(A \otimes B, A' \otimes B').
\]
Composing this with the natural inclusion $A' \otimes B' \rightarrow (A \otimes B)'$, we obtain the external product
\[
H^*(A, A') \times H^*(B, B') \rightarrow H^*(A \otimes B, (A \otimes B)'),
\]
which is also denoted by $\vee$ by abuse of notation.

\begin{prop}
\label{prop:I-is-hom-for-cup-and-vee}
 Let $\phi \in \HC^m(A)$ and $\psi \in \HC^n(B)$.  Then we have $I(\phi \cup \psi) = I(\phi) \vee I(\psi)$, where $I\colon \HC^*(A) \to H^*(A, A')$ is the standard map induced by the identification
 $$
 H^*(A, A') \simeq H^*(\bBcplx(C_*(A))' / \bBcplx(C_*(A))'[2]).
 $$
\end{prop}

\begin{proof}
We let $C_*$ and $D_*$ denote the cyclic modules $C_*(A)$ and $C_*(B)$.  The cup product in cyclic theory was related to the diagram
\[
\xymatrix{
0 & \ar[l] ((C \otimes D)', b^t) & \ar[l]_{I} \Tot \bBcplx(C \otimes D)' & \ar[l]_{S}  \Tot \bBcplx(C \otimes D)'[2] & \ar[l] 0 \\
0 & \ar[l] \ar[u]^{\sh^t} ((C \times D)', b^t) & \ar[l]^{I} \ar[u]^{\Sh^t} \Tot \bBcplx(C \times D)' & \ar[l]^{S} \ar[u]^{\Sh^t} \Tot \bBcplx(C \times D)'[2] & \ar[l] 0
},
\]
see~\cite{MR1600246}*{proof of Theorem~4.3.8}.  The vertical arrows are quasi-isomorphisms.

Let $\HC^*(\Sh^t)$ denote the isomorphism induced by $\Sh^t$ on the cohomology of the total complexes, and similarly $\HH^*(\sh^t)$ be the isomorphism of the $b^t$-complexes induced by $\sh^t$.

On the one hand, the image of $\phi \otimes \psi \in \Tot \bBcplx(C \otimes D)'$ under $\HC^*(\Sh^t)^{-1}$ at the middle column represents $\phi \cup \psi \in \HC^*(A \otimes B)$.  On the other hand, $\HH^*(\sh^t)^{-1}$ at the left column gives the $\vee$-product in Hochschild cohomology~\cite{MR0077480}*{pp.~218--219}. 

Moreover, the image of $\phi \otimes \psi$ under $I$ in the upper row is equal to $I(\phi) \otimes I(\psi)$ by the definition of $C \otimes D$.  Hence we obtain the assertion.
\end{proof}

\subsection{Action of Hochschild cochains on the normalized \texorpdfstring{$(b, B)$}{(b, B)}-bicomplex}

We mostly adopt the convention of Getzler~\cite{MR1261901}, but make as much simplifications as possible.  In particular, we start from ungraded algebras, and we do not consider the brace operation on Hochschild cochains.  We reproduce some of the computations in~\cite{MR1261901} adopted to our drastically simplified setting.

Let $A$ be a unital algebra over $\C$.  The graded space of Hochschild cochains with values in $A$ is defined by
\[
C^n(A, A) = \Lin(A^{\otimes n}, A),
\]
endowed with the Hochschild differential.  When $D \in C^n(A, A)$, we write $\absolute{D} = n - 1$.  There is a subcomplex $\nC^*(A, A)$, the normalized Hochschild cochain complex, consisting of the cochains satisfying
\[
 D(a^1, \ldots, a^{j-1}, 1, a^{j+1}, \ldots, a^n) = 0 \quad (1 \le j \le n).
\]

When $D_1 \in C^m(A, A)$ and $D_2 \in C^n(A, A)$, their \textit{pre-Lie product} $D_1 \circ D_2$ in $C^{m + n}(A, A)$ is defined by the formula
\begin{equation*}
D_1 \circ D_2(a^1, \ldots, a^{m + n}) = \sum_{j=0}^m (-1)^{j |D_2|} D_1(a^1, \ldots, D_2(a^{j+1}, \ldots, a^{j+n}), \ldots, a^{m+n}).
\end{equation*}
The subcomplex $\nC^*(A, A)$ becomes a subalgebra for this product structure.  The product structure of $A$ can be considered as an element $m$ of $C^2(A, A)$, and the associativity of $m$ can be rephrased as $m \circ m = 0$.

The \textit{Gerstenhaber bracket} is given by the supercommutator
\[
[D_1, D_2]_G = D_1 \circ D_2 - (-1)^{\absolute{D_1} \absolute{D_2}} D_2 \circ D_1.
\]
The Hochschild differential $\delta$ on $C^*(A, A)$ can be written as
\[
\delta(D) = m \circ D - (-1)^{\absolute{D}} D \circ m = [m, D]_G.
\]
For example, an element $D \in C^1(A, A)$ is a derivation with respect to $m$ if and only if it satisfies $\delta(D) = 0$.

Let $D$ be a cochain in $\nC^m(A, A)$.  There is an operator $\Bop_D$ from $\nC_n(A)$ to $\nC_{n-m + 2}(A)$ defined by
\begin{multline*}
% \label{eq:Bop-single-D-defn}
\Bop_D(a^0, \ldots, a^n) \\
= \sum_{\substack{0 \le j \le k \\ \le n-m}} (-1)^{n(j+1) + |D|(k-j)} (1, a^{j+1}, \ldots, D(a^{k+1}, \ldots, a^{k+m}), \ldots, a^n, a^0, \ldots, a^{j}).
\end{multline*}
This can be regarded as an analogue of $B$.  Note that $\Bop_D$, like $B$ itself, does not involve the product structure of $A$.

Next, there is another operator $\bop_{D}$ from $\nC_n(A)$ to $\nC_{n-m}(A)$ defined by
\begin{equation*}
% \label{eq:bop-single-D-defn}
\bop_{D}(a^0, \ldots, a^n) = (-1)^{m n}(D(a^{n-m + 1}, \ldots, a^n) a^0, a^1,\ldots, a^{n-m}).
\end{equation*}
We put $\iota_D = \bop_D - \Bop_D$, and call it the \textit{interior product by} $D$.

Finally, the \textit{Lie derivative} $\Lop_D\colon \nC_n(A) \rightarrow \nC_{n-m+1}(A)$ of $D$ is defined by
\begin{multline}
\label{eq:L-D-formula}
\Lop_{ D } (a^0, \ldots, a^n) = \sum_{j=0}^{n-m} (-1)^{|D| (j+1)}(a^0, \ldots, D(a^{j+1}, \ldots, a^{j+m}), \ldots, a^n) \\
+ \sum_{j=n-m+1}^n (-1)^{n(n- j)} (D(a^{j+1}, \ldots, a^0, a^1, \ldots, a^{j+m-n-1}), a^{j+m-n}, \ldots, a^j).
\end{multline}

The above operations are related by the following formula.

\begin{prop}[Cartan homotopy formula~\citelist{\cite{MR0154906}\cite{MR1261901}}]
% \label{prop:cartan-homotopy-formula}
Let $D$ be a cochain in $\nC^*(A, A)$.  We then have
\begin{equation}
\label{eq:cartan-homotopy-formula}
[b - B, \iota_{D}] = \Lop_{D} - \iota_{\delta(D)},
\end{equation}
where the bracket on the left hand side is the graded commutator in $\End(\nC_*(A))$.
\end{prop}

\begin{proof}
We write down the computation for the case $D \in \nC^k(A, A)$ when $k$ is even.  The odd case is proved in a similar way.

Since $\iota_D$ is an even operator on $\nC_*(A)$, the graded commutator on the left hand side of~\eqref{eq:cartan-homotopy-formula} becomes the usual commutator.  In the following computation, $a^0$ denotes a variable on $A$, and $a^1, \ldots, a^n$ are the ones on $A/\C$.

First, one computes $b \bop_D (a^0, \ldots, a^n)$ as
\begin{equation}
\label{eq:b-bD-comp}
\begin{split}
& \sum_{j=0}^{n-k-1} (-1)^{j} (D(a^{n-k+1}, \ldots, a^n) a^0, \ldots, a^j a^{j+1}, \ldots, a^{n-k})\\
 &\quad + (-1)^{n-k}(a^{n-k}D(a^{n-k+1}, \ldots, a^n) a^0, \ldots, a^{n-k-1})\\
&= \bop_D b(a^0, \ldots, a^n) -  \sum_{j=n-k}^{n-1} (-1)^{j} (D(a^{n-k}, \ldots, a^j a^{j+1}, \ldots, a^n)a^0, \ldots, a^{n-k-1})\\
 & \quad - (-1)^n (D(a^{n-k}, \ldots, a^{n-1}) a^n a^0, \ldots, a^{n-k-1}) \\
 & \quad - (-1)^{n+1} (a^{n-k}D(a^{n-k+1}, \ldots, a^n) a^0, \ldots, a^{n-k-1})\\
&= \big ( \bop_D b - \bop_{m \circ D} - \bop_{D \circ m} \big )(a^0, \ldots, a^n).
\end{split}
\end{equation}

Next, $b \Bop_D(a^0, \ldots, a^n)$ can be computed as the sum of
\begin{gather}
\label{eq:b-BD-comp-1}
 (-1)^{n (j+1) + p - j} (a^{j+1},\ldots, D(a^{p+1}, \ldots, a^{p+k}), \ldots, a^n, a^0, \ldots, a^{j}),\\
\label{eq:b-BD-comp-2}
(-1)^{n (j+1) + p - i} (1, a^{j+1}, \ldots, a^i a^{i+1}, \ldots, D(a^{p+1}, \ldots, a^{p+k}), \ldots, a^n, a^0, \ldots, a^{j-1})
\end{gather}
for $j+1 \le i \le p-1$,
\begin{gather}
\label{eq:b-BD-comp-3}
(-1)^{n (j + 1)} (1, a^{j+1}, \ldots, a^{p} D(a^{p+1}, \ldots, a^{p+k}), \ldots, a^n, a^0, \ldots, a^{j}),\\
\label{eq:b-BD-comp-4}
(-1)^{n (j + 1) - 1}(1, a^{j+1}, \ldots, D(a^{p+1}, \ldots, a^{p+k}) a^{p+k+1}, \ldots, a^n, a^0, \ldots, a^j),\\
\label{eq:b-BD-comp-5}
(-1)^{n (j + 1) + p + k - 1 - i} (1, a^{j+1} \ldots, D(a^{p+1}, \ldots, a^{p+k}), \ldots, a^i a^{i+1}, \ldots, a^n, a^0, \ldots, a^j)
\end{gather}
for $p+k+1 \le i \le n-1$,
\begin{gather}
\label{eq:b-BD-comp-6}
(-1)^{n( j + 1) + p + k - 1 - n} (1, a^{j+1} \ldots, D(a^{p+1}, \ldots, a^{p+k}), \ldots, a^n a^0, \ldots, a^j),\\
\label{eq:b-BD-comp-7}
(-1)^{n (j + 1) + p + k - n - i} (1, a^{j+1} \ldots, D(a^{p+1}, \ldots, a^{p+k}),  \ldots, a^n, a^0,\ldots,a^i a^{i+1}, \ldots, a^{j})
\end{gather}
for $0 \le i \le j-1$, and
\begin{equation}
\label{eq:b-BD-comp-8}
(-1)^{n(j + 1) + p + k - n - j} (a^j, \ldots, D(a^{p+1}, \ldots, a^{p+k}), \ldots, a^n, a^0, \ldots, a^{j-1}),
\end{equation}
where $j = 0, \ldots, n - k$ and $p = j, \ldots, n - k$.  The exponent of $-1$ in~\eqref{eq:b-BD-comp-8} is equal to $n j + p - (j - 1) - 1$.  Hence, if we start from $j=n-k$ and gradually decrease $j$, the terms of the form~\eqref{eq:b-BD-comp-8} are cancelled by those of the form~\eqref{eq:b-BD-comp-1} for $j < p$ in the next iteration.  The terms~\eqref{eq:b-BD-comp-8} at $j = 0$ gives
\begin{multline}
\label{eq:where-LD-pops-out}
\sum_{p=0}^{n-k}(-1)^{p}(a^0, \ldots, D(a^{p+1}, \ldots, a^{p+k}), \ldots, a^{n}) = - \Lop_D(a^0, \ldots, a^n) \\
+ \sum_{p=n-k+1}^n (-1)^{n(n-p)} (D(a^{p+1}, \ldots, a^n, a^0, \ldots, a^{p+k-n-1}), a^{p+k-n},\ldots, a^j).
\end{multline}
The terms~\eqref{eq:b-BD-comp-3} and~\eqref{eq:b-BD-comp-4} add up to
\begin{equation}
\label{eq:where-BmD-pops-up}
- \Bop_{m \circ D}(a^0, \ldots, a^n)
 + \sum (-1)^{n(j+1)-1}(1, a^{j+1}, \ldots, D(a^{n-k+1}, \ldots, a^n) a^0, \ldots, a^j).
\end{equation}

For $\Bop_D b (a^0, \ldots, a^n)$, we get the contributions of
\begin{equation}
\label{eq:BD-b-comp-1}
(-1)^{i + (n-1) (j + 1) + p-j}(1, a^{j+1}, \ldots, D(a^{p+1}, \ldots, a^{p+k}), \ldots, a^i a^{i+1},\ldots, a^n, a^0, \ldots, a^j)
\end{equation}
for $p + k < i$,
\begin{equation}
\label{eq:BD-b-comp-2}
(-1)^{i + (n-1) (j + 1) + p - j} (1, a^{j+1}, \ldots, D(a^{p+1}, \ldots, a^i a^{i+1}, \ldots, a^{p+k+1}), \ldots, a^n, a^0, \ldots, a^j)
\end{equation}
for $j + 1 \le i \le p+k$,
\begin{equation}
\label{eq:BD-b-comp-3}
(-1)^{i + (n - 1) (j + 1) + p - j - 1}(1, a^{j+1}, \ldots, a^i a^{i+1}, \ldots, D(a^{p+1}, \ldots, a^{p+k}), \ldots, a^n, a^0, \ldots, a^j)
\end{equation}
for $j+1 \le i \le p-1$,
\begin{equation}
\label{eq:BD-b-comp-4}
(-1)^{i + (n - 1) j + p - j} (1, a^{j+1}, \ldots, D(a^{p+1}, \ldots, a^{p+k}), \ldots, a^n, a^0, \ldots, a^i a^{i+1},  \ldots, a^j)
\end{equation}
for $0 \le i \le j-1$, and
\begin{equation}
\label{eq:BD-b-comp-5}
(-1)^{n + (n-1) (j+1) + p - j} (1, a^{j+1}, \ldots, D(a^{p+1}, \ldots, a^{p+k}), \ldots, a^{n-1}, a^n a^0, \ldots, a^{j}).
\end{equation}
The exponent of $-1$ in~\eqref{eq:BD-b-comp-1} differs from the one in~\eqref{eq:b-BD-comp-5} by an even integer.  Similarly,~\eqref{eq:BD-b-comp-3} is equal to~\eqref{eq:b-BD-comp-2},~\eqref{eq:BD-b-comp-4} is equal to~\eqref{eq:b-BD-comp-7}, and~\eqref{eq:BD-b-comp-5} is equal to~\eqref{eq:b-BD-comp-6}.  The exponent of $-1$ in~\eqref{eq:BD-b-comp-2} is equal to $n(j+1) + i-(p+1) \pmod{2}$.  Hence the sum of these terms is equal to $\Bop_{D \circ m}(a^0, \ldots, a^n)$.

Summarizing the above computations, the effect of $- b \Bop_D + \Bop_D b$ on $(a^0, \ldots, a^n)$ can be expressed as the sum of~\eqref{eq:b-BD-comp-1} for $j = p$,~\eqref{eq:where-LD-pops-out},~\eqref{eq:where-BmD-pops-up}, which is equal to
\begin{multline}
\label{eq:comm-b-BD}
\big ( \Lop_D + \Bop_{m \circ D} + \Bop_{D \circ m} \big )(a^0, \ldots, a^n)\\
+ \sum (-1)^{n (j + 1)}(1, a^{j+1}, \ldots, D(a^{n-k+1}, \ldots, a^n)a^0, \ldots, a^j) \\
+ \sum_{p=n-k+1}^n(-1)^{n(n-p)+1}(D(a^{p+1}, \ldots, a^n, a^0, \ldots, a^{p+k-n-1}), a^{p+k-n}, \ldots, a^j)\\
+ \sum_{j=0}^{n-k} (-1)^{n(j+1) + 1}(D(a^{j+1}, \ldots, a^{p+k}), \ldots, a^n, a^0, \ldots, a^j).
\end{multline}

Next, we compute $B \bop_D(a^0, \ldots, a^n)$ as
\begin{multline*}
(1, D(a^{n-k+1}, \ldots, a^n) a^0, \ldots, a^{n-k}) \\
+ \sum_{i=1}^{n-k} (-1)^{n + (n-k)i }(1, a^{n-k-i+1}, \ldots, D(a^{n-k+1}, \ldots, a^n) a^0, \ldots, a^{n-k-i})
\end{multline*}
The exponent of $-1$ is equal to $n + n i \equiv n + n (n - k + i + 1) \pmod{2}$.  Hence this cancels out with the second part of~\eqref{eq:comm-b-BD}.

Next, $\bop_D B(a^0, \ldots, a^n)$ can be computed as
\[
\sum_{j=0}^n (-1)^{n j} (D(a^{n-j-k+1}, \ldots, a^{n-j}), a^{n-j+1}, \ldots, a^n, a^0, \ldots, a^{n-j-k})),
\]
This cancels out with the last two parts of~\eqref{eq:comm-b-BD}.

Finally, we have $B \Bop_D = 0 = \Bop_D B$ because we are considering the action on the normalized chains

Combining~\eqref{eq:b-bD-comp}, \eqref{eq:comm-b-BD}, and the above paragraphs, we obtain
\[
(b - B) (\bop_D - \Bop_D) - (\bop_D - \Bop_D) (b - B) = - \bop_{m \circ D + D \circ m} + \Lop_D + \Bop_{m \circ D + D \circ m}.
\]
Since $|D|$ is odd, we have $\delta(D) = m \circ D + D \circ m$.  Hence we obtain the desired equality~\eqref{eq:cartan-homotopy-formula}. 
\end{proof}

\subsection{Gauss--Manin connection}

Let $A$ be a vector space, and $(m_\nu)_{\nu \in I^k}$ be a smooth family of associative algebra structures on $A$.  We denote by $A_\nu$ the algebra $(A, m_\nu)$, and by $A_{I^k}$ the space $C^\infty(I^k) \otimes A$ endowed with the `pointwise product structure' $A_\nu$ at $\nu \in I^k$.  We let $\partial_j$ denote the partial derivative for the $j$-th coordinate on $I^k$, for $1 \le j \le k$.  For the operations on the Hochschild cochains, we write $b^{(\nu)}$, $\delta^{(\nu)}$, $\bop^{(\nu)}$, $\iota^{(\nu)}$, and so on when we want to indicate which algebra structure on $A$ we use.

By the associativity of $m_\nu$, we know that $\gamma_\nu = m_\nu - m_0$ is an element of $\nC^2(A, A)$, called \textit{the Maurer--Cartan element for $m_\nu$ with respect to $m_0$}.   This satisfies the following \textit{Maurer--Cartan equation}
\begin{equation}
\label{eq:Maurer-Cartan}
\delta^{(0)}(\gamma_\nu) + \frac{1}{2}[\gamma_\nu, \gamma_\nu]_G = 0.
\end{equation}
We can also take a different origin other than $0$ to define $\gamma_\nu$.  The (partial) derivatives such as $\partial_j \gamma_\nu$ do not depend on the choice of origin.

One of the crucial consequences of~\eqref{eq:Maurer-Cartan} is $\delta^{(0)}(\partial_j \gamma_0) = 0$ for $1 \le j \le k$, which can be obtained by applying $\partial_j$ on the both hand sides of~\eqref{eq:Maurer-Cartan} and evaluating at $\nu = 0$.  Similarly, choosing other origin, we obtain that
\begin{equation}
\label{eq:der-of-MC-is-Hochs-cocycle}
\delta^{(\nu)}(\partial_j \gamma_\nu) = 0
\end{equation}
for any $\nu$ and $j$.

Then, the formula
\[
\nabla(f) = \sum_j \left ( \partial_jf + \iota^{(\nu)}_{\partial_j \gamma_\nu} f \right ) d \nu_j
\]
defines a connection $\nabla$ on the periodic cyclic complex of the algebras $(A_\nu)_\nu$ called the Gauss--Manin connection~\cite{MR1261901}*{Section~3}.  The relation
\[
[b - B, \partial_j] = - \Lop_{\partial_j \gamma_\nu}
\]
and the Cartan homotopy formula for $\partial_j \gamma_\nu$ shows that this operator commutes with $b - B$.  Hence it defines a connection on the periodic cyclic homology group.

The monodromy of this connection defines a canonical isomorphism of the periodic cyclic homology groups under the formal deformation quantization of symplectic manifolds~\cite{MR1350407}*{Appendix~2}.  In general, there is a issue of convergence to define the monodromy operator on the whole periodic cyclic complex, see for example~\cite{MR1667686}*{Section~8}.

The Gauss--Manin connection has the following significance in relation to the evaluation map $\ev_*\colon \HP_*(A_{I^k}) \rightarrow (\HP_*(A_\nu))_{\nu \in I^k}$.

\begin{prop}
The sections in the image of $\ev_*$ are flat with respect to the Gauss--Manin connection.
\end{prop}

\begin{proof}
Indeed, the operator $D = \partial_j$ on $A_{I^k}$ satisfies $\delta(D) = - \partial_j \gamma_\nu$.  Meanwhile $\Lop_{D}$ is the canonical extension of $\partial_j$ to $C_*(A_{I^k})$.  Hence~\eqref{eq:cartan-homotopy-formula} implies
\[
\partial_j + \iota_{\partial_j \gamma_\nu} = [b - B, \iota_{D}]
\]
as an equality between operators acting on $\CC_*(A_{I^k}) = \nbBcplx(A_{I^k})$.  It follows that the connection operator vanishes on $\HP_*(A)$.
\end{proof}

\begin{cor}
 Let $e$ be a projection of $A_{I^k}$.  Then the Chern character of $e$ gives a section of $(\HP_0(A_\nu))_{\nu \in I^k}$ which is flat with respect to the Gauss--Manin connection.
\end{cor}

\subsection{Fell bundle and its \texorpdfstring{$2$}{2}-cocycle deformation}

Let $\Gamma$ be a discrete group.  A \textit{Fell bundle over $\Gamma$} is given by an algebra $A$ together with a $\Gamma$-grading $A = \bigoplus_{g \in \Gamma}$.  Here, we assume the compatibility condition $1 \in A_e$ and $A_g A_h \subset A_{g h}$ between the grading and the algebra structure.  Example of such algebras include group algebra of $\Gamma$, or the crossed product of $\Gamma$ with a $\Gamma$-algebra, and the algebras with torus action for $\Gamma = \Z^n$.

Let $\omega$ be a normalized $\U(1)$-valued $2$-cocycle on $\Gamma$.  The $2$-cocycle property of $\omega$ is given by the equation
\[
\omega(g, h) \omega(g h, k) = \omega(g, h k) \omega(h, k).
\]
The normalization condition on $\omega$ is give by
\[
\omega(g, e) = \omega(e, g) = \omega(g, g^{-1}) = 1.
\]
Any $2$-cocycle is cohomologous to a normalized cocycle.  In the following, we only consider the normalized ones.

The \textit{$\omega$-deformation} $A_\omega$ of $A$ is the algebra with the same underlying linear space as $A$, but endowed with a twisted product
\[
x *_\omega y = \omega(g, h) x \cdot_A y \quad (x \in A_g, y \in A_h).
\]
When $a$ is an element of $A$, we let $a^{(\omega)}$ denote the corresponding element of $A_\omega$.

If $A$ has an antilinear involution $*$, the deformed algebra $A_\omega$ also becomes a $*$-algebra by the same map $*$.  This construction generalizes both the twisted crossed product of $\Gamma$-algebras and the $\theta$-deformation of $T^n$-algebras.

\section{Algebra structure on group cyclic cohomology}
\label{sec:alg-str-on-cyc-coh-of-grp}

\begin{prop}[cf.~\cite{MR913964}*{Lemme~5.18}]
\label{prop:nat-trans-chm-degr-lower-triv}
Let $n > r$ be nonnegative integers, and
\[
\phi_{\Gamma_0, \Gamma_1}\colon H^{n}(\Gamma_0 \times \Gamma_1) \rightarrow H^{r}(\Gamma_0 \times \Gamma_1)
\]
be a natural transformation of functors on the direct product of the category of discrete groups with itself.  Then $\phi$ is zero.
\end{prop}

\begin{proof}
When $\Gamma$ is a discrete group, we represent $\clsB \Gamma$ by a cube complex and denote by $\clsB \Gamma^{(k)}$ its $k$-skeleton for $k \in \N$.  By the Kan--Thurston theorem~\cite{MR0413089}, there is a discrete group $\Lambda(\Gamma, k)$ and a continuous map $\clsB \Lambda(\Gamma, k) \rightarrow \clsB \Gamma^{(k)}$ which induces an isomorphism in cohomology.

Let $s$ and $t$ be nonnegative integers satisfying $s + t = n$.  On one hand, $H^s(\Gamma_0) \rightarrow H^s\big (\clsB \Gamma_0^{(s)}\big) = H^s(\clsB \Lambda(\Gamma_0, s))$ is injective, and on the other hand it is induced by a group homomorphism $\Lambda(\Gamma_0, s) \rightarrow \Gamma_0$.  It follows that we have a commutative diagram
\[
\xymatrix@C=6em{
H^n(\Gamma_0 \times \Gamma_1) \ar[r]^{\phi_{\Gamma_0, \Gamma_1}} \ar[d] & H^{r}(\Gamma_0 \times \Gamma_1) \ar[d]\\
H^n(\Lambda(\Gamma_0, s) \times \Lambda(\Gamma_1, t)) \ar[r]^{\phi_{\Lambda(\Gamma_0, s), \Lambda(\Gamma_1, t)}} & H^{r}(\Lambda(\Gamma_0, s) \times \Lambda(\Gamma_1, t)).
}
\]
By the K\"{u}nneth formula, the group
\[
H^n(\Lambda(\Gamma_0, s) \times \Lambda(\Gamma_1, t)) = \bigoplus_{a + b = n} H^a\left(\clsB \Gamma_0^{(s)}\right) \otimes H^b\left(\clsB \Gamma_1^{(t)}\right)
\]
has to be trivial because we either have $a > s$ or $b > t$.

Since we have $H^{r}(\Lambda(\Gamma_0, s) \times \Lambda(\Gamma_1, t)) = H^s(\Gamma_0) \otimes H^t(\Gamma_1)$, the above argument shows that the $(s, t)$-component $H^n(\Gamma_0 \times \Gamma_1) \rightarrow H^s(\Gamma_0) \otimes H^t(\Gamma_1)$ of $\phi_{\Gamma_0, \Gamma_1}$ is trivial.  Because $(s, t)$ was chosen arbitrarily, $\phi_{\Gamma_0, \Gamma_1}$ has to be also trivial.
\end{proof}

\begin{prop}
The external product on the groups $H^*(\Gamma; \C) = \Ext_{\Gamma}^*(\C, \C)$ and the external cup product on the ones on $\HC^*(\clsB_*\Gamma)$ coincide.
\end{prop}

\begin{proof}
Let $m, n$ be integers, and consider the composition
\begin{multline}
\label{eq:prod-map-from-cyclic-cup}
H^m(\Gamma_0) \otimes H^n(\Gamma_1) \rightarrow \HC^{m}(\clsB_*\Gamma_0) \otimes \HC^{n}(\clsB_*\Gamma_1) \\
\rightarrow \HC^{m+n}(\clsB_* (\Gamma_0 \times \Gamma_1)) = \bigoplus_{k=0}^{\lfloor (m+n)/2 \rfloor}H^{m+n-2k}(\Gamma_0 \times \Gamma_1)
\end{multline}
where the first map is the tensor product of Karoubi's embedding and the middle map is the cup product in cyclic cohomology.

The composition of this map with the projection $H^{m+n}(\Gamma_0 \times \Gamma_1) \rightarrow H^m(\Gamma_0) \otimes H^n(\Gamma_1)$ is a natural transformation in $(\Gamma_0, \Gamma_1)$.  By Proposition~\ref{prop:nat-trans-chm-degr-lower-triv}, the component of $H^{m+n-2k}(\Gamma_0 \times \Gamma_1)$ in~\eqref{eq:prod-map-from-cyclic-cup} is trivial if $k > 0$.

It remains to identify the map $H^m(\Gamma_0) \otimes H^n(\Gamma_1) \rightarrow H^{m+n}(\Gamma_0 \times \Gamma_1)$ induced by~\eqref{eq:prod-map-from-cyclic-cup}.  By definition, the projection $\HC^k(\clsB_*\Gamma) \rightarrow H^k(\Gamma)$ can be identified with $I\colon \HC^k(\clsB_*\Gamma) \rightarrow \HH^k(\clsB_*\Gamma)$.  Thus, by Proposition~\ref{prop:I-is-hom-for-cup-and-vee}, the external cup product of cyclic cohomology is intertwined to the $\vee$-product of $\HH^*$.

Finally, by taking the standard resolution $\C[\Gamma]$ as a $\Gamma$-bimodule, we can see that the $\vee$-product on $\HH^*$ and the external product on $\Ext_{\Gamma}^*(\C, \C)$ coincide.
\end{proof}

The coproduct homomorphism $\Delta\colon\Gamma \rightarrow \Gamma \times \Gamma$ induces a morphism of cyclic sets $\clsB_*(\Gamma) \rightarrow \clsB_*(\Gamma \times \Gamma)$.  Now, the $(b, B)$-bicomplex of $\clsB_*(\Gamma \times \Gamma)$ can be identified with the mixed complex $\clsB_*(\Gamma) \times \clsB_*(\Gamma)$.  Hence the composition of the cup product
\[
\cup\colon \HC^*(\clsB_*\Gamma) \times \HC^*(\clsB_*\Gamma) \rightarrow \HC^*(\clsB_*(\Gamma \times \Gamma))
\]
and the pullback by $\Delta$ defines an associative product structure on $\HC^*(\clsB_*\Gamma)$, which we shall call the \textit{internal cup product}.  Since $\Delta$ is invariant under the flip map, this cup product is graded commutative.

The above algebra structure can be extended to $\HC^*(\C[\Gamma])$ in a straightforward way.

\begin{cor}
The internal cup product on $\HC^*(\clsB_*\Gamma)$ and the cup product on $H^*(\Gamma)$ coincide.
\end{cor}

\begin{rmk}
We also note that if we embed $H^m(\Gamma_0)$ in $\HC^{m+2k}(\clsB_*\Gamma_0)$ and $H^n(\Gamma_1)$ in $\HC^{n + 2l}(\clsB_*\Gamma_1)$, their cyclic cohomology cup product
\[
\tilde{\phi} \cup \tilde{\psi} \in \HC^{m + n + 2(k+ l)}(\clsB_*(\Gamma_0 \times \Gamma_1))
\]
is represented by the image $\widetilde{\phi \cup \psi}$ of group cohomology cup product.  In order to see this, the $S$-operator $\HC^a(A) \rightarrow \HC^{a+2}(A)$ agrees with the map $\phi \mapsto v \cup \phi$, where $v \in \HC^2(\C)$ is given by the cyclic $2$-cocycle on $\C$ characterized by $v(1, 1, 1) = -1/2$~\cite{MR1600246}*{Section~4.4.10}.  By the associativity and the graded commutativity of the cup product, we have
\[
S(\phi \cup \psi) = S(\phi) \cup \psi = \phi \cup S(\psi).
\]
On the other hand, the $S$-operator on $\HC^*(\clsB_*\Gamma)$ is the collection of embeddings $H^k(\Gamma) \oplus H^{k-2}(\Gamma) \oplus \cdots$ into $H^{k+2}(\Gamma) \oplus H^{k}(\Gamma) \oplus \cdots$ which is the identity map on each direct summand.  It follows that $\HP^*(\C[\Gamma])$ has an algebra structure by cup product and $H^*(\Gamma)$ becomes its subalgebra by Karoubi's inclusion.
\end{rmk}

\section{Deformation of cocycles on Fell bundles}

Throughout this section $A = \bigoplus_{g \in \Gamma} A_g$ denotes a $\Gamma$-graded algebra.  We let $\alpha\colon A \rightarrow A \otimes \C[\Gamma]$ be the coaction of the algebraic compact quantum group $(\C[\Gamma], \Delta)$ corresponding to the grading on $A$.

\subsection{Action of group cohomology}

Using the coaction $\alpha$, we can define the action of $\HC^*(\C[\Gamma])$ on $\HC^*(A)$ by
\[
\phi \cdot \omega = \alpha^\#(\phi \cup \omega)\quad(\omega \in \HC^*(\C[\Gamma]), \phi \in \HC^*(A)).
\]
This defines an $\HC^*(\C[\Gamma])$-module structure on $\HC^*(A)$ with respect to the algebra structure on $\HC^*(\C[\Gamma])$ of Section~\ref{sec:alg-str-on-cyc-coh-of-grp}.

A cyclic $n$-cocycle $\phi$ on $A$ is said to be \textit{$\alpha$-invariant} if it satisfies
\[
\phi(f^0, \ldots, f^n) = 0 \quad (f^j \in A_{g_j}, g_0 \cdots g_n \neq e).
\]
Such cocycles are the main subject of our study.  We note that the standard trace $\tau \in \HC^0(\C[\Gamma])$ acts as the idempotent whose image is the $\alpha$-invariant classes of $\HC^*(A)$.  By means of the algebra inclusion $H^*(\Gamma; \C) \rightarrow \HC^*(\C[\Gamma])$, we have an action of group cocycles on the $\alpha$-invariant part of $\HC^*(A)$.

\subsection{Deformation of cyclic cocycles}

The $\omega$-deformation $A_\omega$ admits a homomorphism
\begin{equation*}
% \label{eq:deform-hom-into-orig-otimes-twist-grp-alg}
\alpha_\omega\colon A_\omega \rightarrow A \otimes \C[\Gamma]_\omega, \quad a^{(\omega)} \mapsto a \otimes \lambda^{(\omega)}_g \quad (a \in A_g).
\end{equation*}

Let $\tau$ be the standard trace on $\C[\Gamma]_\omega$.  When $\phi$ is an $\alpha$-invariant cyclic cocycle on $A$, the cup product of $\phi$ and $\tau$ is a cyclic cocycle on $A \otimes \C[\Gamma]_\omega$.  By pulling back along the homomorphism $\alpha_\omega$, we obtain a cyclic cocycle on $A_\omega$ which we denote by $\phi^{(\omega)}$.

\begin{rmk}
There is an alternative definition of $\phi^{(\omega)}$.  Consider a cycle $(\Omega^*, \tau_\phi)$ over $A$ whose character is $\phi$.  We can choose $\Omega^*$ so that the coaction $\alpha$ extends to $\Omega^*$ in an equivariant way, such that the differential $d$ commutes with the coaction.  For example, we can take the universal DGA $\Omega^*(A)$ over $A$.  The $\alpha$-invariance of $\phi$ means that $\tau_\phi(\xi) = 0$ whenever $\xi$ is in the subspace $\Omega^n_g$ for $g \neq e$.

The $\omega$-deformation of $\Omega^*$ becomes a differential graded algebra with the same $d$.  Then, the linear map $\Omega_\omega^n \simeq \Omega^n \rightarrow \C$ is a closed graded trace on $\Omega^*_\omega$: the closedness is trivially satisfied by construction.  The graded trace property follows from the fact that
\[
\tau_\phi(\xi *_\omega \eta - \eta *_\omega \xi) = \omega(g, h) \tau_\phi(\xi \eta) - \omega(h, g) \tau_\phi(\eta \xi) \quad (\xi \in \Omega^*_g , \eta \in \Omega^*_{h})
\]
is zero if $h \neq g^{-1}$ by the $\alpha$-invariance, and is also zero when $h = g^{-1}$ by the cocycle property of $\omega$.  Thus, we obtain a cycle $(\Omega^*_\omega, \tau_{\phi})$ over $A_\omega$, and its character is equal to $\phi^{(\omega)}$;
\begin{equation}
\label{eq:dfn-twisted-cocycle}
\phi^{(\omega)}(f^0, \ldots, f^n) = \tau_\phi(f^0 *_\omega d f^1 *_\omega \cdots *_\omega d f^n).
\end{equation}
\end{rmk}

\subsection{Action of the Maurer--Cartan element}

Let $\omega_0$ be a normalized $\R$-valued $2$-cocycle on $\Gamma$, and put $\omega^t = e^{\sqrt{-1} t \omega_0}$.  Then we obtain a field of algebras $(A_{\omega^t})_{t \in I}$ over $I$.  When $\phi$ is an $\alpha$-invariant cyclic $n$-cocycle on $A$, $\phi^{(t)} = \phi^{(\omega^t)}$ defines a section of the bundle $(\HP^*(A_{\omega^t}))_{t \in I}$.  Pulling back via $\ev_*$ to $\HP^*(A_I)$, we obtain a $1$-parameter family of cyclic cocycles on $A_I$.  By the formula~\eqref{eq:dfn-twisted-cocycle}, each cocycle $\phi^{(t)}$ is rather easy to compute once we know $\phi$.  The cohomology classes $[\phi^{(t)}]$ need not be constant, and the variation can be measured by the monodromy operator of the Gauss--Manin connection.

Let $f^j$ ($j = 1, 2$) be elements respectively in the subspaces $A_{g_j}$ for $g_j \in \Gamma$.  Then the action of Maurer--Cartan element is given by
\begin{align*}
\gamma_t(f_1, f_2) &= (\omega^t(g_1, g_2) - 1) f_1 f_2,&
\partial_t \gamma_t(f_1, f_2) &= \sqrt{-1} \omega_0(g_1, g_2) f_1 f_2.
\end{align*}

For each $n$, consider the function $\omega_0^{(n)} \colon \Gamma^{n+1} \rightarrow \C$ defined by
\[
\omega_0^{(n)}(g_0, \ldots, g_n) = \sum_{j = 1}^n \omega_0(g_0 \cdots g_{j-1}, g_j).
\]
By the cocycle condition on $\omega_0$, we may take any way to parenthesize the product $g_0 \cdots g_n$ and still get the same value from a corresponding formula.  For example, we have the equality
\[
\omega_0^{(3)}(g_0, \ldots, g_3) = \omega_0(g_1, g_2) + \omega_0(g_1 g_2, g_3) + \omega_0(g_0, g_1 g_2 g_3)
\]
corresponding to the associativity $((x y) z) w = x ((y z) w)$.

When the elements $g_0, \ldots, g_n \in \Gamma$ satisfy $g_0 \cdots g_n = e$, we have 
\[
\sqrt{-1} \omega_0^{(n)}(g_0, \ldots, g_n) = \partial_t \tau \big (\lambda^{(\omega^t)}_{g_0} \cdots \lambda^{(\omega^t)}_{g_n} \big ).
\]
This implies the invariance under the cyclic permutation
\begin{equation}
\label{eq:omega-n-inv-under-cyc-perm}
\omega_0^{(n)}(g_0, \ldots, g_n) = \omega_0^{(n)}(g_n, g_0, g_1, \ldots, g_{n-1})
\end{equation}
when $g_0 \cdots g_n = e$.

We let $\omega_0^{(n)} \phi$ denote the functional on $A^{\otimes n}$ characterized by
\[
\omega_0^{(n)} \phi(f^0, \ldots, f^n) = \omega_0^{(n)}(g_0, \ldots, g_n) \phi(f^0, \ldots, f^n)
\]
when $f^j$ is a homogeneous element with spectrum $g_j$ for $0 \le j \le n$.  This satisfies the cyclicity condition by~\eqref{eq:omega-n-inv-under-cyc-perm}, and in fact is the derivative of $\phi^{(t)}$ at $t = 0$.

Since $\phi^{(t)}$ can be considered as a cochain in the normalized $(b, B)$-bicomplex of $A_{\omega^t}^+$, we may consider the action of $\iota_{\partial_t\gamma_t}$ on it.

Consider an infinitesimal perturbation $\delta t$ of the variable $t$.  If we pull back $\phi^{(\delta t)} \in \HP^*(A_{\delta t})$ to $\HP^*(A)$ by the infinitesimal monodromy of the Gauss--Manin connection, we obtain
\[
\phi + \delta t \left( \sqrt{-1} \omega_0^{(n)} \phi - \iota_{\partial_t \gamma_0} \phi \right ) + O(\delta t^2).
\]
It follows that the normalized $(b, B)$-cochain
\begin{equation*}
% \label{eq:int-prod-naive}
(\sqrt{-1} \omega_0^{(n)} - \iota_{\partial_t \gamma_0}) \phi
\end{equation*}
measures the non-flatness of the section $(\phi^{(t)})_t$ at $t = 0$.  Since we have $\Bop_{\partial_t \gamma} \phi = 0$, the cochain $- \iota_{\partial_t \gamma} \phi$ can be written as
\[
\psi = -\phi(\partial_t\gamma(f^{n+1}, f^{n+2})(f^0 + \lambda), f^1, \ldots, f^n), \quad A^+ \otimes A^{\otimes n+2} \rightarrow \C.
\]
Let us put $\psi_{n+2}$ and $\psi_{n + 3}$ the functionals on $A^{\otimes n + 2}$ and $A^{\otimes n + 3}$ corresponding to $\psi$.

It follows that the cochain $(\psi_{n+3}, \psi_{n+2}, \sqrt{-1} \omega_0^{(n)} \phi)$ in $\CC^*(A)$ is a cocycle by the above consideration, but this can be directly verified as follows.  As we already have the cyclicity of $\omega_0^{(n)} \phi$, it remains to show
\begin{gather}
\label{eq:b-B-int-prod-naive-1}
\sqrt{-1} b \omega_0^{(n)} \phi = N \psi_{n+2},\\
\label{eq:b-B-int-prod-naive-2}
b' \psi_{n+2} = (1 - \lambda_{n+2}) \psi_{n+3}, \\
\label{eq:b-B-int-prod-naive-3}
b \psi_{n+3} = 0.
\end{gather}
Firstly, taking the derivative for $t$ of the identity $b \phi_t = 0$ gives~\eqref{eq:b-B-int-prod-naive-1}.  Next, evaluating $b \phi = 0$ on
\[
(\partial_t \gamma(f^{n+1}, f^{n+2}), f^0, f^1, \ldots, f^n),
\]
and combining~\eqref{eq:der-of-MC-is-Hochs-cocycle}, we get~\eqref{eq:b-B-int-prod-naive-2}.  Similarly,~\eqref{eq:b-B-int-prod-naive-3} is a consequence of $b \phi = 0$ evaluated on
\[
(\partial_t \gamma(f^{n+2}, f^{n+3}) f^0, f^1, \ldots, f^{n+1})
\]
and~\eqref{eq:der-of-MC-is-Hochs-cocycle}.

Let us denote by $i_{\sqrt{-1} \omega_0} \phi$ the above cocycle in the cyclic bicomplex of $A$.  This way we obtain an operator $i_{\sqrt{-1} \omega_0}$ acting on the space of $\alpha$-invariant cyclic cocycles on $A_0$.

\begin{prop}
\label{prop:action-maurer-cartan-group-cocycle}
The cyclic cocycles $i_{\sqrt{-1} \omega_0} \phi$ and $\sqrt{-1} \phi \cdot \omega_0$ define the same class in $\HC^*(A)$.
\end{prop}

\begin{proof}
We show the equality between the both divided by $\sqrt{-1}$.  In order to simply the notation, let us denote the cyclic modules $C_*(A^+)$ and $C_*(\clsB \Gamma)$ by $C_*$ and $D_*$ respectively.

We let $\epsilon^\perp$ denote the projection $A^+ \rightarrow A$.  Let $\psi_{n+2}$ be the element of $(C \times D)_{n+2}'$ defined by
\begin{multline*}
\psi_{n+2}((f_0, \ldots, f_{n+2}) \otimes (g_1, \ldots, g_{n+2}))\\
 = - \omega_0(g_{n+1}, g_{n+2}) \phi(\epsilon^\perp(f_{n+1}) \epsilon^\perp(f_{n+2}) f_0, \epsilon^\perp(f_1), \ldots, \epsilon^\perp(f_n)),
\end{multline*}
for $f^j \in A^+$ and $g^j \in \Gamma$.  Similarly, let $\psi_n$ be the element of $(C \times D)_n'$ defined by
\[
\psi_n((f_0, \ldots, f_n) \otimes (g_1, \ldots, g_n)) = \omega_0^{(n)}(g_0, g_1, \ldots, g_n) \phi(\epsilon^\perp(f_0), \ldots, \epsilon^\perp(f_n)),
\]
where we put $g_0 = (g_1 \cdots g_n)^{-1}$.  These cochains satisfy the normalization condition on the part of $C$, and $\psi_{n+1}$ is also normalized for the last two variables on $D$.

Then we claim that $\psi = (\psi_{n+2}, \psi_n)$ is a cocycle in the dual mixed complex $(C \times D)'$.  Indeed, $b \psi_{n+2} = 0$ follows from
\begin{multline*}
- b \psi_{n+2}(f_0 \otimes g_0, \ldots, f_{n+3} \otimes g_{n+3}) \\
= \sum_{j=0}^{n - 1} (-1)^j \omega_0(g_{n+2}, g_{n+3}) \phi(f_{n+2} f_{n+3} f_0, \ldots, f_j f_{j+1}, \ldots, f_{n+1})\\
 + (-1)^n \omega_0(g_{n+1} g_{n+2}, g_{n+3}) \phi(f_{n+1} f_{n+2} f_{n+3} f_0, f_1, \ldots, f_n) \\
+ (-1)^{n+1} \omega_0(g_{n+1}, g_{n+2} g_{n+3}) \phi(f_{n+1} f_{n+2} f_{n+3} f_0, f_1, \ldots, f_n)\\
 + (-1)^{n+2} \omega_0(g_{n+1}, g_{n+2}) \phi(f_{n+1} f_{n+2} f_{n+3} f_0, f_1, \ldots, f_n) \\
= \omega_0(g_{n+2}, g_{n+3}) (b \phi)(f_{n+2} f_{n+3} f_0, \ldots, f_n, f_{n+1}) = 0
\end{multline*}
and $B \psi_n = 0$ is a consequence of $\epsilon^\perp(1) = 0$.

To see that the remaining equality $-B(\psi_{n+2}) + b(\psi_n) = 0$ holds, we observe
\begin{multline*}
b (\psi_n)(f_0 \otimes g_0, \ldots, f_{n+1} \otimes g_{n+1}) \\
= \sum_{j=0}^n (-1)^j \omega_0^{(n)}(g_0, \ldots, g_j g_{j+1}, \ldots, g_{n+1}) \phi(f_0, \ldots, f_j f_{j+1}, \ldots, f_{n+1})\\
 + (-1)^{n+1} \omega^{(n)}(g_{n+1} g_0, \ldots, g_n) \phi(f_{n+1} f_0, \ldots , f_n) \\
= \sum_{j=0}^n \big( \omega_0^{(n+1)}(g_0, \ldots, g_{n+1}) - \omega_0(g_j, g_{j+1}) \big ) \phi(f_0, \ldots, f_j f_{j+1}, \ldots, f_{n+1})\\
 + (-1)^{n+1} \big (  \omega_0^{(n+1)}(g_{n+1}, g_0, \ldots, g_n) - \omega_0(g_{n+1}, g_0) \big ) \phi(f_{n+1} f_0, \ldots, f_n).
\end{multline*}
By the cyclic invariance of $\omega_0^{(n+1)}$, this is equal to
\begin{multline*}
\omega_0^{(n+1)}(g_0, \ldots, g_{n+1}) (b \phi)(f_0, \ldots, f_{n+1}) \\
- \big ( \sum_{j=0}^n (-1)^j \omega_0(g_j, g_{j+1}) \phi(f_0, \ldots, f_j f_{j+1}, \ldots, f_{n+1})\\
 + (-1)^{n+1} \omega_0(g_{n+1}, g_0) \phi(f_{n+1} f_0, \ldots, f_n) \big ).
\end{multline*}
The cyclicity of $\phi$ implies
\[
(-1)^j  \phi(f_0, \ldots, f_j f_{j+1}, \ldots, f_{n+1}) = (-1)^{(n+1) j} \phi(f_j f_{j + 1}, f_{j+2}, \ldots, f_{n+1}, f_0, \ldots, f_{j-1}).
\]
Thus, we obtain that $b(\psi_n)$ is equal to
\[
- \sum_{j=0}^{n+1} (-1)^{(n+1)j} \omega_0(g_j, g_{j+1}) \phi(f_j f_{j+1},  f_{j+2}, \ldots, f_{j-1}),
\]
where we put $g^{n+2} = g^0$.  We can also compute
\[
B \psi_{n+2}(f_0, \ldots, f_{n+1}) = -\sum_{j=0}^{n+1} (-1)^{(n+1) j} \omega_0(g_j, g_{j+1}) \phi(f_j f_{j + 1}, \ldots, f_{n+1}, f_0, \ldots, f_{j-1}),
\]
which implies the desired equality $B \psi_{n+2} = b \psi_n$.

The image of the cocycle $\psi$ under $\Sh^t \colon (C \times D)' \rightarrow (C \otimes D)'$ can be computed as follows.

First, because $\omega_0$ is normalized, $\sh_{n, 2}^\#(\psi_{n+2})$ only contains the contribution from the $(n, 2)$-shuffle $(\ensemble{1, \ldots, n}, \ensemble{n+1, n+2})$.  For the other combinations of $(p, q)$ with $p + q = n+2$, we get $\sh_{p, q}(\psi_{n+2}) = 0$ either by the normalization conditions on the $D$-part of $\psi_{n+2}$ for the last two variables (when $q = 0, 1$) or by the one for the $C$-part (when $p < n$).

Next, $\sh'_{p, q}(\psi_{n+2})$ for $p + q = n$ is always equal to $0$: it is so when $p < n$ because the $C$-part of $\psi_{n+2}$ is normalized, and when $p = n$ because $\omega_0$ is normalized.

Similarly, $\sh_{p, q}(\psi_n)$ for $p + q = n$ and $\sh'_{p, q}(\psi_n)$ for $p + q = n - 2$ are always trivial, either by the normalization condition on the $C$-part of $\psi$ (when $p < n$) or the one on $\omega_0$ (when $q = 0, 1$).

Summarizing the above, we obtain that $\Sh^\#(\psi)$ is represented by the cocycle
\[
\sh_{n, 2}^\#(\psi_{n+2})\big( (f_0, \ldots, f_n) \otimes (g_1, g_2) \big ) = \phi(f_0, \ldots, f_n) \omega_0(g_1, g_2),
\]
which is exactly equal to $\phi \otimes \tau_{\omega_0}$.  Thus, $\psi$ represents the external cup product $\phi \cup \tau_{\omega_0}$ and its pullback by $\alpha^\#\colon (C \times D)' \rightarrow \CC^*(A \otimes \C[\Gamma]) \rightarrow C'$ gives $\phi \cdot \omega_0$.

Finally, by construction, the pullback of $\psi$ by $\alpha^\#$ is easily seen to be $i_{\omega_0} \phi$.
\end{proof}

\begin{thm}
\label{thm:gauss-manin-monodromy}
Suppose that $H^*(\Gamma)$ is bounded, and let $\omega_0$ be a normalized $2$-cocycle on $\Gamma$ with values in $\R$.  Furthermore, let $A$ be a $\Gamma$-graded algebra, and let $(c_t)_{t \in I}$ be a family of elements in $(\HP_*(A_{\omega^t}))_{t \in I}$ which is flat with respect to the Gauss--Manin connection.  Then, the number
\begin{equation}
\label{eq:pair-grp-exp-maurer-flat}
\pairing{\phi^{(t)} \cdot e^{t \sqrt{-1} [\omega_0]}}{c_t}
\end{equation}
is independent of $t \in I$.
\end{thm}

\begin{proof}
If we take the time derivative of~\eqref{eq:pair-grp-exp-maurer-flat}, we obtain
\[
\sqrt{-1} \pairing{(\phi^{(t)} \cdot e^{t \sqrt{-1} [\omega_0]}) \cdot \omega_0}{c_t} + \pairing{\phi^{(t)} \cdot e^{t \sqrt{-1} [\omega_0]}}{\partial_t c_t}.
\]
The flatness of $(c_t)_t$ means $\partial_t c_t = \iota_{\omega_0} c_t$.  Applying Proposition~\ref{prop:action-maurer-cartan-group-cocycle} to $A_{\omega^t}$ and $\phi^{(t)} \cdot e^{t \sqrt{-1} [\omega_0]}$, we obtain
\[
\pairing{\phi^{(t)} \cdot e^{t \sqrt{-1} [\omega_0]}}{\partial_t c_t} = -\sqrt{-1} \pairing{\left (\phi^{(t)} \cdot e^{t \sqrt{-1} [\omega_0]} \right ) \cdot \omega_0}{c_t},
\]
which implies the assertion.
\end{proof}

\begin{cor}
\label{cor:compar-deform-cocycle-action-e-iomega}
Let $c \in \HP_*(A_I)$, and $\phi$ be a cyclic cocycle on $A$.  Then we have
\[
\pairing{\phi \cdot e^{-\sqrt{-1} [\omega_0]}}{c_0} = \pairing{\phi^{(\omega)}}{c_1}.
\]
\end{cor}

\begin{proof}
 Apply the above theorem to the $\alpha$-invariant cocycle $\phi \cdot e^{-\sqrt{-1} [\omega_0]}$, and use the fact that $e^{-\sqrt{-1}[\omega_0]} \cup e^{\sqrt{-1}[\omega_0]} = [\tau]$ acts as the projection onto the space of $\alpha$-invariant classes.
\end{proof}

\begin{rmk}
The assumption on the boundedness of $H^*(\Gamma)$ is needed to represent $e^{\sqrt{-1} [\omega_0]}$ as a class in $\HP^*(\C[\Gamma])$.  Even without it, if $\omega_0$ satisfies a certain boundedness condition, we may make sense of it in the entire cyclic cohomology theory $\mathrm{HE}^*(\C[\Gamma])$ for a suitable topology/bornology on $\C[\Gamma]$.  For example, this is the case if $\Gamma$ satisfies the Polynomial Cohomology condition and the Rapid Decay condition~\cite{MR1066176}.
\end{rmk}

\begin{rmk}
Let $c$ be a $\Gamma$-invariant $\R$-valued $2$-cocycle on the universal proper $\Gamma$-space $\propEx \Gamma$.  Since there is a $\Gamma$-equivariant map $E \Gamma \rightarrow \propEx \Gamma$, the pullback of $c$ defines a class in $H^2_\Gamma(E \Gamma) = H^2(B \Gamma) = H^2(\Gamma)$, which we call $\omega$.

Let $(M, E)$ be a geometric cycle for $\Gamma$; that is, $M$ is a proper cocompact $\Gamma$-manifold with a $\Gamma$-equivariant spin$^c$ structure, and $E$ is an $\Gamma$-equivariant Hermitian vector bundle over $M$.  Then we obtain an element $x$ of $K_*(C^*_{r, \omega}(\Gamma))$ as the image of $(M, E)$ under the twisted version of the Baum--Connes assembly map.  Mathai's result~\cite{MR2218025}*{Section~5.2} says that
\[
\tau(x) = \frac{1}{\sqrt{2 \pi}^{\dim M}} \int_{M/\Gamma} \mathrm{Todd}(M) \wedge e^{f^*(c)} \wedge \ch(E),
\]
where $f$ is a (homotopically unique) $\Gamma$-equivariant map from $M$ to $\propEx \Gamma$.  Corollary~\ref{cor:compar-deform-cocycle-action-e-iomega}, applied to the special case of $A = \C[\Gamma]$ and $\phi = \tau$, can be interpreted as a purely algebraic description of this formula.
\end{rmk}

\section{Flow invariant cocycles}

Let $\smooth{A}$ be a Fr\'{e}chet algebra, $\alpha$ be a smooth action of $\R$ on $\smooth{A}$, and $D$ be the generator of $\alpha$. If $\alpha$ is an action of $\R/\Z$, this is the same thing as giving a $\Z$-grading on $\smooth{A}$.

Let $\phi$ be a cyclic $n$-cocycle on $\smooth{A}$.  There is another formulation of interior product of $D$ and $\phi$ due to Connes~\cite{MR823176}, defined by
\begin{equation}
\label{eq:interior-prod-deriv-cyclic-cocycle-defn}
i_{D}\phi(f^0, \ldots, f^{n+1}) = \frac{1}{n+1} \sum_{j=1}^{n+1} (-1)^j \pairing{\phi}{f^0 d f^1 \cdots D(f^j) \cdots d f^{n+1}}.
\end{equation}
Each summand on the right hand side is a Hochschild cocycle which is cohomologous to $\iota_{D}^t \phi$ in $(\CC^*(\smooth{A}) , b + b') \subset (\nbBcplx(\smooth{A}^+)', b^t)$.  Indeed, for $1 \le k \le n$, put
\[
\psi^{(k)}(f^0, \ldots, f^n) = \phi(f^0, \ldots, D(f^k), \ldots, f^n).
\]
Then we compute $b \psi^{(k)}(f^0, \ldots, f^{n+1})$ as
\begin{multline*}
\sum_{j = 0}^{k-1} (-1)^j \phi(\ldots, f^j f^{j+1}, \ldots, D(f^{k+1}), \ldots) + (-1)^k \phi(\ldots, D(f^k f^{k+1}), \ldots) \\
+ \sum_{j=k+1}^n (-1)^j \phi(\ldots, D(f^k), \ldots, f^j f^{j+1}, \ldots) + (-1)^{n+1} \phi(f^{n+1} f^0, \ldots, D(f^k), \ldots)\\
= (-1)^k \phi(\ldots, D(f^k) f^{k+1}, \ldots) + \sum_{j=k+1}^n (-1)^j \phi(\ldots, D(f^k), \ldots, f^j f^{j+1}, \ldots) \\
+ (-1)^{n+1} \phi(f^{n+1} f^0, \ldots, D(f^k), \ldots) + (-1)^k \phi(\ldots, D(f^{k+1}) f^{k+2}, \ldots)\\
 + \sum_{j = k+2}^n (-1)^{j-1} \phi(\ldots, D(f^{k+1}), \ldots, f^j f^{j+1}, \ldots) + (-1)^{n} \phi(f^{n+1} f^0, \ldots, D(f^k), \ldots)\\
= (-1)^k \left ( \pairing{\phi}{f^0 d f^1 \cdots D(f^k) \cdots d f^{n+1}} + \pairing{\phi}{f^0 d f^1 \cdots D(f^{k+1}) \cdots d f^{n+1}} \right ).
\end{multline*}
Thus, the right hand side of~\eqref{eq:interior-prod-deriv-cyclic-cocycle-defn} is a Hochschild cocycle cohomologous to
\[
\phi \vee D(f^0, \ldots, f^{n+1}) = \phi(D(f^{n+1}) f^0, \ldots, f^n) = \iota_D^t \phi(f^0, \ldots, f^{n+1}).
\]
When $\phi$ is $\alpha$-invariant, $i_{D}\phi$ satisfies the cyclicity condition.  This way we obtain an operator $i_{D}$ acting on the $\alpha$-invariant cyclic cocycles.

The above construction can be explained as an operation on the $\alpha$-invariant normalized $(b, B)$-cochains using the Cartan homotopy formula as follows.

\begin{prop}
\label{prop:connes-vs-getzler-int-prod-der}
Let $\phi$ be an $\alpha$-invariant cocycle in the cyclic bicomplex of $\smooth{A}$.  Then $\iota_{D} \phi$ defines a cocycle which is cohomologous to $i_{D} \phi$.
\end{prop}

\begin{proof}
In view of~\eqref{eq:L-D-formula}, the $\alpha$-invariance of $\phi$ can be expressed as $\Lop_{D} \phi = 0$.  Since $D$ is a derivation, we also have $\delta(D) = 0$.  Thus,~\eqref{eq:cartan-homotopy-formula} implies that $\iota_{D} \phi$ is a $(b, B)$-cocycle.  It is represented by the functional on $\Omega^{n+1}(\smooth{A})$ defined by
\begin{align*}
\psi_0(f^0, \ldots, f^{n+1}) &= \phi(D(f^{n+1}) f^0, \ldots, f^n), \quad \smooth{A}^{n+2} \rightarrow \C,\\
\psi_1(f^1, \ldots, f^{n+1}) &= \phi(D(f^{n+1}), f^1, \ldots, f^n), \quad \smooth{A}^{n+1} \rightarrow \C.
\end{align*}
If we put
\begin{equation}
\label{eq:cob-between-i-D-and-iota-D}
\Psi(f^0, \ldots, f^n) = \sum_{j = 1}^n (-1)^{n j} j \phi(D(f^j), f^{j+1}, \ldots, f^n, f^0, f^1, \ldots, f^{j-1}),
\end{equation}
we have $(1 - \lambda) \Psi = (n+1)\psi_1$ by the $\alpha$-invariance of $\phi$.  Hence $\iota_{D} \phi$ is cohomologous to $\psi_0 - b \Psi / (n+1)$, which is a cyclic cocycle.  This cocycle is equal to $i_{D} \phi$.
\end{proof}

\begin{prop}
\label{prop:i-D-sq-zero}
Let $\phi$ be an $\alpha$-invariant cyclic $n$-cocycle.  Then the class of $i_D^2 \phi$ is trivial in $\HC^{n+2}(\smooth{A})$.
\end{prop}

\begin{proof}
We first note that the cochain~\eqref{eq:cob-between-i-D-and-iota-D} is $\alpha$-invariant.  Again by the Cartan homotopy formula for $D$, we conclude that $\iota_D^2 \phi$ and $i_D^2 \phi$ are in the same cohomology class.

The effect of $\iota_D^2$ on $\phi$ can be written as
\begin{multline}
\label{eq:iota-D-effect-on-phi}
\phi(D(f^{n+1}) D(f^{n+2}) f^0, f^1, \ldots, f^n) \\
+ \sum_{0 \le j \le k \le n - 1} (-1)^{n (j+1) + (n + 1) + 1}\phi(D(f^j), f^{j+1}, \ldots, D(f^{k+1}), \ldots, f^{j-1}).
\end{multline}

Consider a normalized Hochschild cochain $D \cup D\colon (a^0, a^1) \mapsto D(a^0) D(a^1)$.  This is a coboundary of $\frac{1}{2} D \circ D$.  The first part of the above formula is equal to $\bop_{D \cup D} \phi = \iota_{D \cup D} \phi$.  By the cyclicity of $\phi$, the second part is equal to
\[
\sum_{0 \le j \le k \le n-1} \phi(f^0, \ldots, D(f^j), \ldots, D(f^{k+1}), \ldots, f^n).
\]
The $\alpha$-invariance of $\phi$ implies that this is equal to $-\frac{1}{2} \Lop_{D \circ D} \phi$.  Thus,~\eqref{eq:iota-D-effect-on-phi} is equal to the effect of $\iota_{D \cup D} - \frac{1}{2} \Lop_{D \circ D}$ on $\phi$, which is null homotopic by the Cartan homotopy formula.
\end{proof}

\section{\texorpdfstring{$\theta$}{theta}-deformation}

We consider the $\theta$-deformation of $\R^2$-algebras introduced by Rieffel~\cite{MR1002830}.  Let $\smooth{A}$ be a Fr\'{e}chet algebra, and $\sigma$ be a smooth action of $\R^2$ on $\smooth{A}$.  We let $D_1$ and $D_2$ denote the generators of $\sigma$ into each direction.

When $\theta$ is a real number, the deformed product $*_\theta$ on $\smooth{A}$ is defined as the oscillatory integral (see also~\cite{arXiv:1207.2560}*{Appendix})
\begin{equation}
\label{eq:theta-dfm-R2-osc-int}
f^0 *_\theta  f^1 = \int_{\R^4} e^{\sqrt{-1} (x x' + y y')} \sigma_{(x, y)}(f^0) \sigma_{(-\theta y', \theta x')}(f^1) d x d y d x' d y'.
\end{equation}
We let $\smooth{A}_\theta$ denote the algebra $(\smooth{A}, *_\theta)$.

Regarding $\theta$ as a coordinate on $I$, we obtain a bundle of algebras $(\smooth{A}_\theta)_{\theta \in I}$ over $I$.  Let $\gamma_\theta$ denote the Maurer--Cartan element.  Taking the derivative of the right hand side of \eqref{eq:theta-dfm-R2-osc-int}, we obtain
\begin{multline*}
\partial_\theta \gamma_\theta(f^0, f^1)
= \int_{\R^4} e^{\sqrt{-1} (x x' + y y')} \sigma_{(x, y)}(f^0)
\\ \left ( x' \sigma_{(-\theta y', \theta x')}(D_2 f^1) - y' \sigma_{(-\theta y', \theta x')}(D_1 f^1) \right ) d x d y d x' d y'.
\end{multline*}
Integrating by parts, we see that the right hand side is equal to
\begin{equation*}
% \label{eq:maurer-cartan-theta-dfrm}
D_1(f^0) *_\theta D_2(f^1) - D_2(f^0) *_\theta D_1(f^1).
\end{equation*}

The universal differential graded algebra $\Omega^*(\smooth{A})$ over $\smooth{A}$ admits an action of $\R$, again denoted by $\sigma$.  The $\theta$-deformation of $\Omega^*(\smooth{A})$ for this action is a DGA over $\smooth{A}_\theta$.

Suppose that $\phi$ is a $\sigma$-invariant cyclic $n$-cocycle on $\smooth{A}$.  Then the  closed graded trace $\tau_\phi \colon \Omega^n(\smooth{A}) \rightarrow \C$ associated to $\phi$ is invariant under $\sigma$.  The map $\Omega^n(\smooth{A})_\theta \rightarrow \C$ induced by the linear space identification $\Omega^n(\smooth{A}) \simeq \Omega^n(\smooth{A})_\theta$ vanishes on the graded commutators on $\Omega^*(\smooth{A})_\theta$, and defines a closed graded trace on it.  Thus, we have a cyclic $n$-cocycle $\phi_\theta$ on $\smooth{A}_\theta$ by
\begin{equation*}
% \label{eq:theta-deform-tw-cocycle-as-trace}
\phi_\theta(f^{0, (\theta)}, \ldots, f^{n, (\theta)}) = \tau_\phi(f^0 *_\theta d f^1 *_\theta \cdots *_\theta d f^n),
\end{equation*}
where $f^0, \ldots, f^n$ are elements of $\smooth{A}$, and $f^{0, (\theta)}, \ldots, f^{n, (\theta)}$ are the corresponding ones in $\smooth{A}_\theta$.

When $f^0$ and $f^1$ are elements in $\smooth{A}$, put
\[
T(f^0, f^1) = D_1(f^0) D_2(f^1) - D_2(f^0) D_1(f^1).
\]
Then, define the operations $(T^{(n)})_{n = 1, 2, \ldots}$ inductively by
\begin{equation*}
T^{(1)} = T, \quad T^{(n+1)}(f^0, \ldots, f^{n+1}) = T^{(n)}(f^0, \ldots, f^n) f^{n+1} + T(f^0 \cdots f^n, f^{n+1}).
\end{equation*}
Let $m^{(n)}(f^0, \ldots , f^n) = f^0 \cdots f^n$.  If we take the derivative of $*_\theta$ with respect to $\theta$, we obtain
\[
T^{(n)}(f^0, \ldots, f^n) = \partial_\theta(f^0 *_\theta \cdots *_\theta f^n).
\]
When $\phi$ is a cyclic $n$-cocycle on $\smooth{A}$, we put
\[
\phi \circ T_d^{(n)}(a^0, \ldots, a^n) = \tau_\phi\big (T^{(n)}\big ( a^0, d a^1, \ldots, d a^{n} \big ) \big ).
\]
When $\sigma$ comes from an action of torus, this is nothing but $\omega_0^{(n)} \phi$ of the previous section.

\begin{thm}[cf.~\cite{MR2738561}*{Theorem~2}]
\label{thm:theta-deform-inv-cocycle-deform}
 Let $\sigma$ be an action of $\R^2$ on $\smooth{A}$, and $\phi$ be a $\sigma$-invariant cyclic $n$-cocycle on $\smooth{A}$.  Let $(c(\theta))_{\theta \in I}$ be a family of periodic cyclic cycles on $\smooth{A}_\theta$ which is flat with respect to the Gauss--Manin connection.  Then the pairing
\begin{equation}
\label{eq:const-pairing-for-theta-deform}
\pairing{\phi_\theta + \theta (i_{D_1} i_{D_2} - i_{D_2} i_{D_1} )  \phi_\theta}{c(\theta)}
\end{equation}
is constant.
\end{thm}

\begin{proof}
We want to show that the time derivative of~\eqref{eq:const-pairing-for-theta-deform} is zero.  By Proposition~\ref{prop:i-D-sq-zero} applied to $D_1$ and $D_2$, we obtain that $i_{D_1}$ and $i_{D_2}$ act as square-zero operators on the invariant classes.  The same lemma applied to $D_1 + D_2$ shows that $i_{D_1}$ and $i_{D_2}$ anticommute at the level of cohomology.  Thus, $i_{D_1} i_{D_2} - i_{D_2} i_{D_1}$ acts as a square-zero operator.  This shows that, by replacing $\phi_\theta$ by $\phi_\theta + \theta (i_{D_1} i_{D_2} - i_{D_2} i_{D_1} )  \phi_\theta$, we may reduce the generic case to $\theta = 0$.

As before, we consider $\phi$ as a cochain in $\nbBcplx(\smooth{A}^+)'$.  By Proposition~\ref{prop:connes-vs-getzler-int-prod-der}, $i_{D_2} \phi$ and $\iota_{D_2} \phi$ are cohomologous.  Moreover, the cochain $\Psi$ between them, as in the proof of Proposition~\ref{prop:connes-vs-getzler-int-prod-der}, can be taken to be $\sigma$-invariant.  By the Cartan homotopy formula and $\Lop_{D_1} \Psi = 0$, we obtain that $\iota_{D_1} i_{D_2} \phi$ and $\iota_{D_1} \iota_{D_2} \phi$ are in the same cohomology class.  Hence, the cyclic cohomology classes of $i_{D_1} i_{D_2} \phi$ and $\iota_{D_1} \iota_{D_2} \phi$ are the same.

We then claim that
\begin{equation}
\label{eq:rep-comm-int-prod-by-mc-action-and-time-der-prod}
(\iota_{D_1} \iota_{D_2} - \iota_{D_2} \iota_{D_1} ) \phi = \iota_{\partial_\theta \gamma_\theta} \phi - \phi \circ T_d^{(n)}.
\end{equation}
First, $\bop_{D_2}\phi$ is represented by a functional on $\Omega^{n+1}(A)$ whose components are given by
\[
\begin{split}
\smooth{A}^{\otimes n + 2} &\rightarrow \C, \quad (a^0, \ldots, a^{n+1}) \mapsto (-1)^{n+2}\phi(D_2(a^{n+1})a^0, a^1, \ldots, a^n),\\
\smooth{A}^{\otimes n + 1} &\rightarrow \C, \quad (a^1, \ldots, a^{n+1}) \mapsto (-1)^{n+2} \phi(D_2(a^{n+1}), a^1, \ldots, a^n).
\end{split}
\]
Thus, we know that $\bop_{D_1} \bop_{D_2} \phi$ is represented by a functional over $\Omega^{n+2}(\smooth{A})$ whose components are given by
\[
\begin{split}
\smooth{A}^{\otimes n + 3} &\rightarrow \C, \quad (a^0, \ldots, a^{n+2}) \mapsto -\phi(D_2(a^{n+1}) D_1(a^{n+2}) a^0, a^1, \ldots, a^n),\\
\smooth{A}^{\otimes n + 2} &\rightarrow \C, \quad (a^1, \ldots, a^{n+2}) \mapsto -\phi(D_2(a^{n+1}) D_1(a^{n+2}), a^1, \ldots, a^n).
\end{split}
\]
Thus, the commutator $[\bop_{D_1}, \bop_{D_2}]$ agrees with $\bop_{\partial_\theta \gamma_\theta}$.\footnote{We can deduce this also from the fact that $\bop$ is a twisting cochain, and that $m\ensemble{D_1, D_2} - m\ensemble{D_2, D_1}$ is equal to $\partial_\theta \gamma_\theta$, see~\cite{MR1261901}.}  Next, $\Bop_{D_1} \bop_{D_2} \phi$ can be written as a functional on $\Omega^{n}(\smooth{A})$ whose component on $\smooth{A}^{n+1}$ given by
\begin{multline*}
\quad \sum_{0 \le j \le k \le n} (-1)^{(n+1) + n(j+1)} \phi(D_2(a^{j}), a^{j+1}, \ldots, D_1(a^{k+1}), \ldots, a^n, a^0, \ldots, a^{j-1}) \\
= \quad \sum_{0 \le j \le k \le n} -\phi(a^0, \ldots, D_2(a^{j}), \ldots, D_1(a^{k+1}), \ldots, a^n),
\end{multline*}
and is trivial on $\smooth{A}^{n}$.  Thus, $(- \Bop_{D_1} \bop_{D_2} + \Bop_{D_2} \bop_{D_1}) \phi$ is equal to $-\phi \circ T_d^{(n)}$.  Since we have $\Bop_{D} \phi = 0$ for any normalized cochain $D$, we obtain~\eqref{eq:rep-comm-int-prod-by-mc-action-and-time-der-prod}.

The pairing~\eqref{eq:const-pairing-for-theta-deform} is equal to
\[
\pairing{\phi_\theta + \theta (\iota_{D_1} \iota_{D_2} - \iota_{D_2} \iota_{D_1})  \phi_\theta}{c(\theta)}.
\]
Taking the derivative with respect to $\theta$, we obtain
\[
\pairing{\phi_\theta \circ T_d^{(n)} + \iota_{\partial_\theta \gamma_\theta}\phi_\theta - \phi_\theta \circ T_d^{(n)}}{c(\theta)} + \pairing{\phi_\theta + \theta (\iota_{D_1} \iota_{D_2} - \iota_{D_2} \iota_{D_1})  \phi_\theta}{- \iota_{\partial_\theta \gamma_\theta} c(\theta)}.
\]
The terms cancel out each other at $\theta = 0$.
\end{proof}

\begin{rmk}
The more general case of $\R^n$-actions (where the deformation parameter is given by a skewsymmetric matrix $(\theta_{i, j})_{i, j}$ of size $n$) can be represented by successive iteration of the $\theta$-deformation of $\R^2$-actions.  The resulting formula is that
\[
\prod_{i < j} \left ( 1  + \theta_{i, j} \left ( i_{D_i} i_{D_j} - i_{D_j} i_{D_i}\right ) \right) \phi_{\theta} = \prod_{i, j} \left (1 + \theta_{i, j} i_{D_i} i_{D_j} \right ) \phi_\theta
\]
is constant when paired with flat sections of periodic cyclic homology.
\end{rmk}

\subsection{Crossed product by a one-parameter group}

The crossed product of actions by $\R$ can be thought as a particular case of $\theta$-deformation as follows.  Suppose that $\alpha$ is an action of $\R$ on $\smooth{A}$.  The smooth crossed product $\smooth{A} \rtimes_\alpha \R$ is given by the Fr\'{e}chet space $\schwfuncs^*(\R; \smooth{A}, \alpha)$ of the Schwartz class functions of $\R$ into $\smooth{A}$ endowed with the convolution product.

Elliott, Natsume, and Nest~\cite{MR945014} showed that there is a canonical map
\begin{equation*}
% \label{eq:ENN-map-between-HC}
\#_\alpha\colon \HC^n(\smooth{A}) \rightarrow \HC^{n+1}(\smooth{A} \rtimes_\alpha \R)
\end{equation*}
which induces an isomorphism between $\HP^{*+1}(\smooth{A})$ and $\HP^*(\smooth{A} \rtimes_\alpha \R)$, which is compatible with the Connes--Thom isomorphism in the $K$-theory.

In~\cite{MR2738561}, we deduced, at least for the case of torus action, the statement of Theorem~\ref{thm:theta-deform-inv-cocycle-deform} from the Elliott--Natsume--Nest isomorphism.  In fact we can reverse the direction and show that this map (for $\alpha$-invariant cyclic cocycles) is compatible with the Connes--Thom isomorphism from Theorem~\ref{thm:theta-deform-inv-cocycle-deform}.

Let $\alpha^s$ be the rescaled action $\alpha^s_t = \alpha_{s t}$.  In the pre-C$^*$-algebraic setting, the evaluation maps
\[
C^\infty(I; \smooth{A} \rtimes_{\alpha^*} \R) \rightarrow \smooth{A} \rtimes_{\alpha^s} \R
\]
for $s \in I$ induce the Connes--Thom isomorphism of the C$^*$-algebraic completions
\[
K_{*+1}(A) \simeq K_*(A \rtimes_{\triv} \R) \simeq K_*(C(I; A \rtimes_{\alpha^*} \R)) \simeq K_*(A \rtimes_\alpha \R).
\]

Now, the crossed product $\smooth{A} \rtimes_{\alpha^s} \R$ has two actions of $\R$, given by
\begin{align*}
\alpha_t(f)_\xi &= \alpha_t(f_\xi), &
\widehat{\alpha^s}_t(f)_\xi &= e^{\sqrt{-1} t \xi} f_\xi
\end{align*}
for $f \in \schwfuncs^*(\R; \smooth{A}, \alpha^s)$ and $t \in \R$.  These two actions commute by construction.  Their generators can be written as
\begin{align*}
D_\alpha(f)_\xi &= D_\alpha(f_\xi), &
\hat{D}(f)_\xi &= \xi f_\xi,
\end{align*}
where $D_\alpha$ is the generator of $\alpha$.  If we take the $\theta$-deformation $\schwfuncs^* \ptensor \smooth{A} = \smooth{A} \otimes_{\triv} \R$ for this action, we obtain the crossed products $\smooth{A} \rtimes_{\alpha^\theta} \R$.

Let $\phi$ be an $\alpha$-invariant cyclic $n$-cocycle on $\smooth{A}$.  Then, there is a cyclic $n$-cocycle $\hat{\phi}$, called the \textit{dual cocycle}, on $\smooth{A} \rtimes_\alpha \R$.  It is defined by
\[
\hat{\phi}(f^0, \ldots, f^n) = \int_{\xi_0 + \cdots + \xi_n = 0} \phi \left (f^0_{\xi_0}, \alpha_{\xi_0}(f^1_{\xi_1}), \ldots, \alpha_{\xi_0 + \ldots + \xi_{n-1}}(f^n_{\xi_n}) \right ),
\]
or equivalently, in the language of closed graded traces, by
\[
\tau_{\hat{\phi}(f^0 d f^1 \cdots d f^n)} = \tau_\phi(f^0 * d f^1 * \cdots * d f^n |_0).
\]
Then, $\hat{\phi}$ is $\hat{\alpha}$-invariant.  We obtain a cyclic $(n+1)$-cocycle
\begin{equation*}
%\label{eq:ENN-isom-for-inv-cocycle-is-interior-prod-for-dual-action}
\#_\alpha(\phi) = i_{\hat{D}} \hat{\phi}
\end{equation*}
on $\smooth{A} \rtimes_\alpha \R$.  We can similarly define the operation $\#_{\alpha^s}$ for each rescaling of $\alpha$.  At $s = 0$, the cocycle $\#_{\alpha^0}$ is cohomologous to $\phi \cup \eta$ where $\eta$ is the cyclic $1$-cocycle on $\schwfuncs^*$ defined by
\[
\eta(f^0, f^1) = \int_{\R} d t f^0_{-t} f^1_t t.
\]
By the Fourier transform, we see that $\eta$ is identified with the fundamental $1$-current on $\schwfuncs(\R)$.

\begin{prop}[cf.~\cite{MR945014}]
% \label{prop:ENN-isom-for-invar-cocycle}
 Let $c(\theta)$ be a cycle in $\hat{\Omega}_*(\Gamma^\infty_I(\smooth{A} \rtimes_{\alpha_\star} \R))$ which defines a flat class with respect to the Gauss--Manin connection on $\HP_*(\Gamma_I^\infty(\smooth{A} \rtimes_{\alpha_\star} \R))$.  Let $\phi$ be an $\alpha$-invariant cyclic $n$-cocycle on $\smooth{A}$.  Then the pairing $\pairing{\#_{\alpha_\theta}\phi}{c(\theta)}$ does not depend on the value of $\theta$.
\end{prop}

\begin{proof}
By Theorem~\ref{thm:theta-deform-inv-cocycle-deform}, we know that $\#_{\alpha^s} \phi + \theta i_{D_\alpha} i_{\hat{D}} \#_{\alpha^s} \phi$ is the monodromy of the Gauss--Manin connection.  By Proposition~\ref{prop:i-D-sq-zero}, the class of $i_{D_\alpha} i_{\hat{D}} \#_{\alpha^s} \phi$ is trivial.
\end{proof}


\begin{bibdiv}
\begin{biblist}

\bib{arXiv:1207.2560}{misc}{
      author={Bhowmick, Jyotishman},
      author={Neshveyev, Sergey},
      author={Sangha, Amandip},
       title={Deformation of operator algebras by {B}orel cocycles},
         how={preprint},
        date={2012},
      eprint={\href{http://arxiv.org/abs/1207.2560}{{\tt arXiv:1207.2560
  [math.OA]}}},
}

\bib{MR0077480}{book}{
      author={Cartan, Henri},
      author={Eilenberg, Samuel},
       title={Homological algebra},
   publisher={Princeton University Press},
     address={Princeton, N. J.},
        date={1956},
      review={\MR{0077480 (17,1040e)}},
}

\bib{arXiv:0906.3122}{misc}{
      author={Cattaneo, Alberto~S.},
      author={Felder, Giovanni},
      author={Willwacher, Thomas},
       title={The character map in deformation quantization},
         how={preprint},
        date={2009},
      eprint={\href{http://arxiv.org/abs/0906.3122}{{\tt arXiv:0906.3122
  [math.QA]}}},
}

\bib{MR1066176}{article}{
      author={Connes, Alain},
      author={Moscovici, Henri},
       title={Cyclic cohomology, the {N}ovikov conjecture and hyperbolic
  groups},
        date={1990},
        ISSN={0040-9383},
     journal={Topology},
      volume={29},
      number={3},
       pages={345\ndash 388},
      review={\MR{MR1066176 (92a:58137)}},
}

\bib{MR823176}{article}{
      author={Connes, Alain},
       title={Noncommutative differential geometry},
        date={1985},
        ISSN={0073-8301},
     journal={Inst. Hautes \'Etudes Sci. Publ. Math.},
      volume={62},
       pages={257\ndash 360},
      review={\MR{MR823176 (87i:58162)}},
}

\bib{MR2762543}{incollection}{
      author={Dolgushev, V.~A.},
      author={Tamarkin, D.~E.},
      author={Tsygan, B.~L.},
       title={Noncommutative calculus and the {G}auss-{M}anin connection},
        date={2011},
   booktitle={Higher structures in geometry and physics},
      series={Progr. Math.},
      volume={287},
   publisher={Birkh\"auser/Springer, New York},
       pages={139\ndash 158},
      review={\MR{2762543 (2011h:19001)}},
}

\bib{MR731772}{incollection}{
      author={Elliott, G.~A.},
       title={On the {$K$}-theory of the {$C\sp{\ast} $}-algebra generated by a
  projective representation of a torsion-free discrete abelian group},
        date={1984},
   booktitle={Operator algebras and group representations, {V}ol. {I}
  ({N}eptun, 1980)},
      series={Monogr. Stud. Math.},
      volume={17},
   publisher={Pitman},
     address={Boston, MA},
       pages={157\ndash 184},
      review={\MR{731772 (85m:46067)}},
}

\bib{MR945014}{article}{
      author={Elliott, George~A.},
      author={Natsume, Toshikazu},
      author={Nest, Ryszard},
       title={Cyclic cohomology for one-parameter smooth crossed products},
        date={1988},
        ISSN={0001-5962},
     journal={Acta Math.},
      volume={160},
      number={3-4},
       pages={285\ndash 305},
         url={http://dx.doi.org/10.1007/BF02392278},
         doi={10.1007/BF02392278},
      review={\MR{MR945014 (89h:46093)}},
}

\bib{MR1261901}{incollection}{
      author={Getzler, Ezra},
       title={Cartan homotopy formulas and the {G}auss-{M}anin connection in
  cyclic homology},
        date={1993},
   booktitle={Quantum deformations of algebras and their representations
  ({R}amat-{G}an, 1991/1992; {R}ehovot, 1991/1992)},
      series={Israel Math. Conf. Proc.},
      volume={7},
   publisher={Bar-Ilan Univ.},
     address={Ramat Gan},
       pages={65\ndash 78},
      review={\MR{1261901 (95c:19002)}},
}

\bib{arXiv:1011.2735}{misc}{
      author={Hassanzadeh, Mohammad},
      author={Khalkhali, Masoud},
       title={Cup coproducts in {H}opf-cyclic cohomology},
         how={preprint},
        date={2010},
      eprint={\href{http://arxiv.org/abs/1011.2735}{{\tt arXiv:1011.2735
  [math.KT]}}},
}

\bib{MR732839}{article}{
      author={Karoubi, Max},
       title={Homologie cyclique des groupes et des alg{\`e}bres},
        date={1983},
        ISSN={0249-6291},
     journal={C. R. Acad. Sci. Paris S{\'e}r. I Math.},
      volume={297},
      number={7},
       pages={381\ndash 384},
      review={\MR{732839 (85g:18012)}},
}

\bib{MR913964}{article}{
      author={Karoubi, Max},
       title={Homologie cyclique et {$K$}-th\'eorie},
        date={1987},
        ISSN={0303-1179},
     journal={Ast\'erisque},
      number={149},
       pages={147},
      review={\MR{913964 (89c:18019)}},
}

\bib{MR2444365}{incollection}{
      author={Kontsevich, Maxim},
       title={X{I} {S}olomon {L}efschetz {M}emorial {L}ecture series: {H}odge
  structures in non-commutative geometry},
        date={2008},
   booktitle={Non-commutative geometry in mathematics and physics},
      series={Contemp. Math.},
      volume={462},
   publisher={Amer. Math. Soc.},
     address={Providence, RI},
       pages={1\ndash 21},
        note={Notes by Ernesto Lupercio},
      review={\MR{2444365 (2009m:53236)}},
}

\bib{MR2112033}{article}{
      author={Khalkhali, Masoud},
      author={Rangipour, Bahram},
       title={Cup products in {H}opf-cyclic cohomology},
        date={2005},
        ISSN={1631-073X},
     journal={C. R. Math. Acad. Sci. Paris},
      volume={340},
      number={1},
       pages={9\ndash 14},
         url={http://dx.doi.org/10.1016/j.crma.2004.10.025},
         doi={10.1016/j.crma.2004.10.025},
      review={\MR{2112033 (2005h:16013)}},
}

\bib{MR0413089}{article}{
      author={Kan, D.~M.},
      author={Thurston, W.~P.},
       title={Every connected space has the homology of a {$K(\pi ,1)$}},
        date={1976},
        ISSN={0040-9383},
     journal={Topology},
      volume={15},
      number={3},
       pages={253\ndash 258},
      review={\MR{0413089 (54 \#1210)}},
}

\bib{MR1600246}{book}{
      author={Loday, Jean-Louis},
       title={Cyclic homology},
     edition={Second},
      series={Grundlehren der Mathematischen Wissenschaften [Fundamental
  Principles of Mathematical Sciences]},
   publisher={Springer-Verlag},
     address={Berlin},
        date={1998},
      volume={301},
        ISBN={3-540-63074-0},
        note={Appendix E by Mar{\'{\i}}a O. Ronco, Chapter 13 by the author in
  collaboration with Teimuraz Pirashvili},
      review={\MR{1600246 (98h:16014)}},
}

\bib{MR2218025}{incollection}{
      author={Mathai, Varghese},
       title={Heat kernels and the range of the trace on completions of twisted
  group algebras},
        date={2006},
   booktitle={The ubiquitous heat kernel},
      series={Contemp. Math.},
      volume={398},
   publisher={Amer. Math. Soc.},
     address={Providence, RI},
       pages={321\ndash 345},
        note={With an appendix by Indira Chatterji},
      review={\MR{2218025 (2007c:58034)}},
}

\bib{MR1350407}{article}{
      author={Nest, Ryszard},
      author={Tsygan, Boris},
       title={Algebraic index theorem},
        date={1995},
        ISSN={0010-3616},
     journal={Comm. Math. Phys.},
      volume={172},
      number={2},
       pages={223\ndash 262},
         url={http://projecteuclid.org/getRecord?id=euclid.cmp/1104274104},
      review={\MR{1350407 (96j:58163b)}},
}

\bib{MR1667686}{incollection}{
      author={Nest, Ryszard},
      author={Tsygan, Boris},
       title={On the cohomology ring of an algebra},
        date={1999},
   booktitle={Advances in geometry},
      series={Progr. Math.},
      volume={172},
   publisher={Birkh\"auser Boston},
     address={Boston, MA},
       pages={337\ndash 370},
      review={\MR{1667686 (99k:16018)}},
}

\bib{MR595412}{article}{
      author={Pimsner, M.},
      author={Voiculescu, D.},
       title={Imbedding the irrational rotation {$C\sp{\ast} $}-algebra into an
  {AF}-algebra},
        date={1980},
        ISSN={0379-4024},
     journal={J. Operator Theory},
      volume={4},
      number={2},
       pages={201\ndash 210},
      review={\MR{595412 (82d:46086)}},
}

\bib{MR2475613}{article}{
      author={Rangipour, Bahram},
       title={Cup products in {H}opf cyclic cohomology via cyclic modules},
        date={2008},
        ISSN={1532-0073},
     journal={Homology, Homotopy Appl.},
      volume={10},
      number={2},
       pages={273\ndash 286},
         url={http://projecteuclid.org/getRecord?id=euclid.hha/1251811077},
      review={\MR{2475613 (2010d:16007)}},
}

\bib{rieffel-irr-rot-pres}{misc}{
      author={Rieffel, Marc~A.},
       title={Irrational rotation {C$^*$}-algebras},
         how={short communication},
        date={1978},
        note={presented at International Congress of Mathematicians, Helsinki},
}

\bib{MR1002830}{article}{
      author={Rieffel, Marc~A.},
       title={Deformation quantization of {H}eisenberg manifolds},
        date={1989},
        ISSN={0010-3616},
     journal={Comm. Math. Phys.},
      volume={122},
      number={4},
       pages={531\ndash 562},
         url={http://projecteuclid.org/getRecord?id=euclid.cmp/1104178588},
      review={\MR{1002830 (90e:46060)}},
}

\bib{MR0154906}{article}{
      author={Rinehart, George~S.},
       title={Differential forms on general commutative algebras},
        date={1963},
        ISSN={0002-9947},
     journal={Trans. Amer. Math. Soc.},
      volume={108},
       pages={195\ndash 222},
      review={\MR{0154906 (27 \#4850)}},
}

\bib{MR2308582}{article}{
      author={Tsygan, Boris},
       title={On the {G}auss-{M}anin connection in cyclic homology},
        date={2007},
        ISSN={1029-3531},
     journal={Methods Funct. Anal. Topology},
      volume={13},
      number={1},
       pages={83\ndash 94},
      review={\MR{2308582 (2008b:16011)}},
}

\bib{MR2738561}{article}{
      author={Yamashita, Makoto},
       title={Connes-{L}andi deformation of spectral triples},
        date={2010},
        ISSN={0377-9017},
     journal={Lett. Math. Phys.},
      volume={94},
      number={3},
       pages={263\ndash 291},
      eprint={\href{http://arxiv.org/abs/1006.4420}{{\tt arXiv:1006.4420}}},
      review={\MR{2738561}},
}

\bib{arXiv:1107.2512}{misc}{
      author={Yamashita, Makoto},
       title={Deformation of algebras associated to group cocycles},
         how={preprint},
        date={2011},
      eprint={\href{http://arxiv.org/abs/1107.2512}{{\tt arXiv:1107.2512
  [math.OA]}}},
}

\end{biblist}
\end{bibdiv}

\end{document}